\documentclass[12pt]{amsart}

\usepackage[T1]{fontenc}
\usepackage{lmodern}

\usepackage{verbatim}
\usepackage{mathtools}
\usepackage{amssymb}

\usepackage[usenames,dvipsnames]{color}

\newcommand{\R}    {{\mathbb R}}
\newcommand{\N}    {{\mathbb N}}
\newcommand{\eps}  {\varepsilon}
\newcommand{\E}    {\operatorname{E}}
\newcommand{\Var}  {\operatorname{Var}} 
\newcommand{\rbit} {\operatorname{rbit}}
\newcommand{\quant}{\operatorname{quant}}
\newcommand{\bp}   {{\boldsymbol{p}}}
\newcommand{\res}  {\mathrm{res}}
\newcommand{\Lip}  {\mathrm{Lip}}
\newcommand{\M}    {\operatorname{\mathfrak M}}
\newcommand{\U}    {\operatorname{\mathfrak U}}
\newcommand{\V}    {\operatorname{\mathfrak R}}
\newcommand{\F}    {\operatorname{\mathfrak F}}
\newcommand{\supp} {\operatorname{supp}}
\newcommand{\spa}  {\operatorname{span}}
\newcommand{\cost} {\operatorname{cost}}

\theoremstyle{plain}
\newtheorem{lemma}{Lemma}
\newtheorem{theo}{Theorem}
\newtheorem{cor}{Corollary}

\theoremstyle{definition}
\newtheorem{rem}{Remark}

\newcommand{\overbar}[1]%
{\mkern 1.5mu\overline{\mkern-1.5mu#1\mkern-1.5mu}\mkern 1.5mu}

\begin{document}

\title%
[Random Bit Quadrature]
{Random Bit Quadrature and Approximation of \\[3pt]
Distributions on Hilbert Spaces}

\author[Giles]
{Michael B.\ Giles}
\address{Mathematical Institute\\
University of Oxford\\
Oxford OX2 6GG\\
England} 
\email{mike.giles@maths.ox.ac.uk} 

\author[Hefter]
{Mario Hefter}
\address{Fachbereich Mathematik\\
Technische Universit\"at Kaisers\-lautern\\
Postfach 3049\\
67653 Kaiserslautern\\
Germany}
\email{\{hefter,lmayer,ritter\}@mathematik.uni-kl.de}

\author[Mayer]
{Lukas Mayer}

\author[Ritter]
{Klaus Ritter}

\begin{abstract}
We study the approximation of expectations $\E(f(X))$
for Gaussian random elements $X$ with values in a separable
Hilbert space $H$ and Lipschitz continuous functionals
$f \colon H \to \R$. 
We consider restricted Monte Carlo
algorithms, which
may only use random bits instead of random numbers.
We determine the asymptotics (in some cases 
sharp up to multiplicative constants, in the other cases sharp up 
to logarithmic factors) of the corresponding $n$-th minimal error
in terms of 
the decay of the eigenvalues of the covariance operator of $X$.
It turns out that,
within the margins from above,
restricted Monte Carlo algorithms are not inferior
to arbitrary Monte Carlo algorithms,
and suitable random bit
multilevel algorithms 
are optimal.
The analysis of this problem 
leads to a variant of the quantization problem, namely,
the optimal approximation of probability measures on $H$ 
by uniform distributions supported by 
a given, finite number of
points.
We determine the asymptotics (up to multiplicative constants)
of the error of the best approximation
for the one-dimensional standard normal distribution,
for Gaussian measures as above,
and for scalar autonomous SDEs.
\end{abstract}

\keywords{%
Gaussian measures on Hilbert spaces, integration, approximation
of probability measures, quantization, random bits,
multilevel Monte Carlo algorithms, stochastic differential
equations}

\subjclass[2010]{60G15, 60H35, 60H10, 65D30, 65C05}

\maketitle

\section{Introduction}

We study the approximation of expectations $\E(f(X))$, where $X$ is a 
random element that takes values in a separable Hilbert space $H$
and where $f \colon H \to \R$ is Lipschitz continuous. We consider 
randomized (Monte Carlo) algorithms
that are only allowed to use random bits instead of 
random numbers. By assumption, all other operations 
(arithmetic operations, evaluations of elementary functions, and oracle 
calls to evaluate $f$) are performed exactly.
Algorithms of this type are called restricted
Monte Carlo algorithms, and the approximation of expectations
by algorithms of this type will be called random bit quadrature.

Let $\mu$ denote the distribution of $X$.
We consider the worst case setting, where randomized
algorithms $A$ are compared according to their maximal
error $e(A,F,\mu)$ and their maximal cost $\cost(A,F)$ on a 
class $F$ of functionals $f$.
For an arbitrary Monte Carlo algorithm or a restricted Monte Carlo 
algorithm $\cost(A,F)$ takes into account, in particular, 
the number of calls of the generator for random numbers or random bits, 
respectively.

A basic question is to what extent restricted Monte Carlo algorithms are 
inferior to arbitrary Monte Carlo algorithms. To answer this
question one has to compare the $n$-th minimal error
\[
e^\res_n(F,\mu)
=\inf\{e(A,F,\mu)\colon 
\text{$A$ restricted Monte Carlo algorithm, 
$\cost(A,F)\leq n$}\}
\]
of restricted Monte Carlo algorithms
with the corresponding quantity $e_n(F,\mu)$ for
arbitrary Monte Carlo algorithms on classes $F$ of functionals 
$f$.
In the case of 
infinite-dimensional spaces $H$, this question is closely related
to three variants of approximation problems for probability
measures, namely, quantization, average Kolmogorov widths, and
random bit approximation. 

In most of the papers on randomized algorithms
for continuous problems, uniformly distributed random numbers 
from $[0,1]$ are assumed to be available.
Restricted Monte Carlo algorithms are studied 
for the classical quadrature problem, where $\mu$ is the
uniform distribution on $[0,1]^d \subseteq H = \R^d$, 
in, e.g., \cite{GYW06,HNP04,N85,N88,N01,TW92,YH08}.

In the present paper, we are interested in zero mean Gaussian random 
elements $X$ with values in a separable Hilbert space $H$ and with a 
distribution $\mu$ with infinite-dimensional support, and in
the class $F = \Lip_1$ of all Lipschitz
continuous functionals $f \colon H \to \R$ with Lipschitz constant at 
most one. 

The Karhunen-Lo\`eve expansion of $X$ may be written as
\[
X = \sum_{i=1}^\infty \lambda_i^{1/2} \cdot Y_i \cdot e_i
\]
with convergence, e.g., in mean-square with respect to the norm of $H$.
Here $e_1, e_2, \dots$ form an orthonormal system in $H$, 
$\lambda_1 \geq \lambda_2 \ldots > 0$ with $\sum_{i=1}^\infty
\lambda_i < \infty$, and
$Y_1,Y_2,\dots$ are independent and standard normally
distributed random variables.
We assume that
\[
\lim_{i\to\infty}
\lambda_i \cdot i^{\beta} \cdot (\ln(i))^{\alpha}
\in \left]0,\infty\right[,
\]
where $\beta > 1$ and $\alpha \in \R$.
The asymptotic behavior of the variances $\lambda_i$ of the random
coefficients of $X$ is known in many cases, see, e.g., \cite{LP04}.
For instance, 
\[
\beta = 2h +1
\]
and 
\[
\alpha= -(d-1)\cdot \beta
\]
for a fractional Brownian sheet $X$ 
with Hurst parameter 
$h \in \left]0,1\right[$ and $H=L_2([0,1]^d)$.
In particular,
$\beta=2$ and $\alpha=0$ for a Brownian motion, as well as for a 
Brownian bridge. 

For functions $f,g \colon M \to \left[0,\infty\right]$
on any set
$M$ we write $f(m) \preceq g(m)$ if
there exists a constant $c>0$ such that $f(m) \leq c \cdot g(m)$
for every $m \in M$. Moreover, 
$f(m) \succeq  g(m)$ means $g(m) \preceq f(m)$ and 
$f(m) \asymp  g(m)$ means $f(m) \preceq g(m)$ and 
$g(m) \preceq f(m)$.
In order to mention the set $M$ explicitly, we sometimes
say that the corresponding relation holds uniformly in $m \in M$.

We show that suitable random bit multilevel Monte Carlo algorithms
yield the upper bound
\[
e^\res_n(\Lip_1,\mu)
\preceq
n^{-\min(1/2,(\beta-1)/2)}
\cdot
\begin{cases}
1,
&\text{if }\beta>2,\\
(\ln(n))^{\max(0,1-\alpha/2)},
&\text{if }\beta=2\ \wedge \alpha\neq 2,\\
\ln(\ln(n)),
&\text{if }\beta=2\ \wedge \alpha= 2,\\
(\ln(n))^{-\alpha/2},
&\text{if }\beta<2,
\end{cases}
\]
for the $n$-th minimal error of restricted Monte Carlo algorithms,
see Theorem \ref{thmmlmc} and Corollary \ref{corquad}.
See \cite{G15} for a recent survey of multilevel algorithms.

Upper and lower bounds for the $n$-th minimal
error of arbitrary Monte Carlo algorithms have been established in
\cite{CDMGR09} in a Banach space setting.
Combining the upper bound from the present paper
with the lower bound from
\cite[Thm.~10]{CDMGR09} we obtain the following results, see
Corollary \ref{corquad}.
For $\beta < 2$
\[
e^\res_n(\Lip_1,\mu) \asymp e_n(\Lip_1,\mu) \asymp
n^{-(\beta-1)/2} \cdot (\ln(n))^{-\alpha/2}
\]
and, in particular, there is
no superiority of Monte Carlo algorithms over
restricted Monte Carlo algorithms in this case. 
For $\beta = 2$, a superiority may at most be present on the level of
logarithmic factors, since
\[
\frac{e^\res_n(\Lip_1,\mu)}{e_n(\Lip_1,\mu)}
\preceq
\begin{cases}
\ln(n)^{\max(1,\alpha/2)}, & \text{if $\alpha\neq 2$,}\\
\ln(n) \cdot \ln(\ln(n)), & \text{if $\alpha=2$.}\\
\end{cases}
\]
For $\beta > 2$ we may only conclude that
\[
\liminf_{n \to \infty}
\frac{e^\res_n(\Lip_1,\mu)}{e_n(\Lip_1,\mu)}
\preceq
(\ln(n))^{(1+\beta)/2} \cdot (\ln(\ln(n)))^{\alpha/2}.
\]
Note that for many infinite-dimensional quadrature problems the asymptotic 
behavior of $n$-th minimal errors is only known up to logarithmic factors.
Except for the case $\beta = 2$ and $\alpha<1$,
the upper bound from the present paper slightly
improves the respective bound from \cite{CDMGR09}, which holds
true in a Banach space setting.

In \cite{HNP04}, random bit quadrature with respect to 
the uniform distribution $\mu$ on $[0,1]^d$
and Sobolev and H\"older classes $F$ of functions on
$[0,1]^d$ are considered. The $n$-th minimal 
errors of unrestricted and of restricted Monte Carlo algorithms
turn out to be of the same order, and a very small number of
$O((2+d) \cdot \log n)$ random bits suffice to achieve
asymptotic optimality. 
The proofs of these results are based
on a reduction of the quadrature problem to a summation problem
and on a discrete variant of Bakhvalov's trick.
See \cite{NP04} for a related approach to integral equations.
Anisotropic function classes are considered in \cite{GYW06,YH08}.

For the Gaussian measures $\mu$ on infinite-dimensional spaces we do not
know whether the number of random bits that are needed to achieve
the upper bound for $e^\res_n(\Lip_1,\mu)$ (or asymptotic
optimality) is negligible, compared to $n$.
In our construction of a multilevel algorithm that yields
the upper bound for $e^\res_n(\Lip_1,\mu)$ the number of
random bits is of the order $n$.

The analysis of random bit quadrature problems leads to
the following variant of the quantization problem for probability
measures, namely, the optimal approximation of probability 
measures $\mu$ by uniform distributions $\nu$ on $2^p$ points.
Since $p$ random bits suffice to sample any such $\nu$, we use the
term random bit approximation of probability measures
to denote this new type of approximation problem. 
Let $d$ denote the Wasserstein distance of order two on
the set $\M(H)$ of all Borel probability measures on $H$, and let
$\U(H,p) \subseteq \M(H)$ denote the set of all uniform distributions on 
$H$ with support of size $2^p$. 
Given $\mu \in \M(H)$ 
we study the distance
\[
\rbit(\mu,p) = \inf\{ d(\mu, \nu) \colon \nu \in \U(H,p)\}
\]
between $\mu$ and $\U(H,p)$.
In the one-dimensional case $H=\R$ this approximation
problem has recently been introduced and thoroughly studied for
Wasserstein distances of any order in \cite{XB17}, and some
of the results from \cite{XB17} are generalized to the Banach space
$\R^d$, equipped with the maximum norm, for any $d \in \N$ in
\cite{C18}.

Random bit approximation is
closely related to quantization, which has been studied
intensively for finite-dimensional and for infinite-dimensional
Banach spaces $H$. More precisely, let
$\F(H,p)$
denote the set of all Borel probability measures on $H$ with support of
size at most $2^p$. Obviously the quantization number
\[
\quant(\mu,p) = \inf\{ d(\mu, \nu) \colon \nu \in \F(H,p)\}
\]
is a lower bound for $\rbit(\mu,p)$, i.e.,
\[
\rbit(\mu,p) \geq \quant(\mu,p)
\]
for every $\mu \in \M(H)$ and every $p \in \N$. 
A partial list of references on quantization of probability
measures includes the monograph \cite{GL00} and the survey
\cite{D09} as well as
\cite{CDMGR09,D03,D07a,D07,DFMS03,DS05,LP02,LP04,LP06}.
We stress that the strong asymptotics of $\quant(\mu,p)$
is studied most of the time in the literature, while we only
consider the weak asymptotics of $\rbit(\mu,p)$. Observe that 
we lose the control about asymptotic constants anyway in the analysis 
of the random bit quadrature problem.

For the one-dimensional standard normal distribution $\mu$ we derive
\[
\rbit(\mu,p) \asymp 2^{-p/2} \cdot p^{-1/2},
\]
see Theorem \ref{p:approx-stnormal-prop}, 
while $\quant(\mu,p) \asymp 2^{-p}$ according to a known
general result for quantization. 
For the Gaussian measures $\mu$ on Hilbert spaces we have 
\[
\rbit(\mu,p) \asymp p^{-(\beta-1)/2} 
\cdot (\ln(p))^{-\alpha/2},
\]
see Theorems \ref{p:Bp-rate} and \ref{n-t1}.
For scalar autonomous SDEs 
we consider the distribution $\mu$ of the solution on $L_2([0,1])$,
and under mild assumptions on the drift and diffusion
coefficients we have
\[
\rbit(\mu,p) \asymp p^{-1/2},
\]
see Theorem \ref{p5}.
In the latter two cases we only have to establish the upper bound 
for $\rbit(\mu,p)$, since the matching lower bound
even holds for $\quant(\mu,p)$, according to known results
for quantization.

In the present paper we employ upper bounds for $\rbit(\mu,p)$
and asymptotically optimal random bit approximations to construct
random bit algorithms for quadrature and to derive upper 
bounds for $e^\res(\Lip_1,\mu)$.
In \cite{CDMGR09} close relations between quantization numbers and
average Kolmogorov widths on the one-hand side, and upper and lower
bounds for (minimal) errors of arbitrary Monte Carlo algorithms
have been established.

This work is partially motivated by
reconfigurable architectures like field programmable gate arrays (FPGAs).
These devices allow users to choose
the precision of each individual operation on a bit level and provide a 
generator for random bits.
In the setting and analysis of the present
paper we take into account the latter fact, while we 
ignore all finite precision issues for arithmetic operations.
We refer to \cite{Betal15a,Betal15b} for the construction and 
for extensive tests of a
finite precision multilevel algorithm for FPGAs with applications
in computational finance. For an error analysis of the Euler scheme
for SDEs in a finite precision arithmetic we refer to \cite{O16}.

This paper is organized as follows. In Section \ref{s2}
we formulate and study the random bit approximation problem
for probability measures. Section \ref{sec3} is devoted to the
analysis of random bit quadrature with respect to Gaussian
measures. In the Appendix we derive some asymptotic properties
of the distribution function and its inverse for the 
standard normal distribution.

\section{Random Bit Approximation of Probability Measures}\label{s2}

\subsection{Definitions and Basic Properties}

Consider the set $\M(V)$ of all Borel probability measures
on a separable Banach space $(V,\|\cdot\|_V)$
with a finite second moment,
equipped with the 
Wasserstein distance $d$ of order two, i.e.,
\[
d(\mu_1,\mu_2) = \inf
\Bigl\{ \bigl( \E \|X_1-X_2\|_V^2 \bigr)^{1/2} \colon 
P_{X_1} = \mu_1,\ P_{X_2} = \mu_2 \Bigr\}
\]
for $\mu_1,\mu_2 \in \M(V)$. Here $X_1$ and $X_2$ are jointly
defined on any probability space
and take values in $V$, and $P_{X_i}$ denotes the
distribution of $X_i$. 

For $p\in\N$ we use $\nu^{(p)}$ to denote the uniform distribution on 
$\{0,1\}^p$, and we define
\begin{align*}
\V(V,p)
=
\{ \nu^{(p)}_f \in \M(V) \colon
f \colon \{0,1\}^p \to V\},
\end{align*}
where $\nu^{(p)}_f$ denotes the distribution of $f$ with respect to 
$\nu^{(p)}$.
Observe that $\V(V,p)$ is the set of all 
probability measures on $V$ with support of size at most $2^p$
and with probability weights being
integer multiples of $1/2^p$.
Clearly $p$ random bits
suffice to sample from any $\nu \in \V(V,p)$.

Given $\mu \in \M(V)$ we study the distance
\[
\rbit(\mu,p) = \inf\{ d(\mu, \nu) \colon \nu \in \V(V,p)\}
\]
between $\mu$ and $\V(V,p)$.
We wish to determine the asymptotic behavior of $\rbit(\mu,p)$
as $p$ tends to infinity
and to construct probability measures $\mu^{(p)} \in \V(V,p)$
such that $d(\mu,\mu^{(p)})$ is close to $\rbit(\mu,p)$.

Specifically, we are interested in separable Hilbert spaces $(H,\|\cdot\|_H)$ and
the cases of $\mu$
being the one-dimensional
standard normal distribution, the distribution of 
a Brownian bridge 
on $H=L_2([0,1])$ or, more generally, of a Gaussian random
element on an infinite-dimensional Hilbert space $H$, and finally
the distribution of the solution of a scalar SDE on $H=L_2([0,1])$.

\begin{rem}\label{r2.1}
Obviously,
\[
\rbit(\mu,p+1)
\leq
\rbit(\mu,p)
\]
for every $\mu \in \M(V)$ and every $p \in \N$.
\end{rem}

\begin{rem}\label{r2.2}
Let
\[
\U(V,p)
=
\{ \nu^{(p)}_f \in \M(V) \colon
f \colon \{0,1\}^p \to V\text{ is injective}\},
\]
which is the set of all uniform distributions on $V$
with support of size $2^p$. Since $\U(V,p) \subseteq \V(V,p)$ 
with a dense embedding with respect to the Wasserstein distance
$d$, we have
\[
\rbit(\mu,p) = \inf\{ d(\mu, \nu) \colon \nu \in \U(V,p)\}
\]
for every $\mu \in \M(V)$ and every $p \in \N$. Consequently,
random bit approximation deals with the optimal
approximation of probability measures by uniform distributions
on $2^p$ points.
\end{rem}

The one-dimensional case $V = \R$ has been thoroughly studied in a 
more general setting in \cite{XB17}, and some of the results in the
latter paper have been generalized in \cite{C18} to the Banach
space $V=\R^d$, equipped with the maximum norm, for any $d \in \N$.
Given $p \in \N$ and probability weights $a_1,\ldots,a_{2^p}$ 
the objective is to minimize the Wasserstein
distance of order $r$ between a Borel probability measure $\mu$ on
$V$ with a finite moment of order $r$ and
$\nu = \sum_{k=1}^{2^p} a_k \cdot \delta_{x_k}$
with Dirac measures $\delta_{x_k}$ at any points $x_k$.
This problem is called best finite constrained approximation
with prescribed weights $a_k$ in \cite{XB17}.

The special case $V=\R$, $r=2$, and $a_k=2^{-p}$ corresponds
to the random bit approximation of $\mu \in \M(\R)$, and
we present key results from \cite{XB17} in this case.
In the sequel, $\Psi^{-1}$ denotes the inverse of the distribution 
function of $\mu$.

\begin{rem}\label{xuberger:bounds}
According to \cite[Rem.~5.6(ii)]{XB17}, the unique best
approximation $\nu \in \V(\R,p)$ of $\mu \in \M(\R)$ with respect to $d$
is determined by the points
\begin{equation}\label{g55}
\phantom{\qquad\quad k=1,\ldots,2^p.}
x^*_k = 2^{p} \cdot 
\int_{(k-1) \cdot 2^{-p}}^{k \cdot 2^{-p}} \Psi^{-1}(t) \, \mathrm{d}t,
\qquad\quad k=1,\ldots,2^p.
\end{equation}

Assume that the measure corresponding
to $\Psi^{-1}$ is absolutely continuous  with respect to the Lebesgue 
measure on $[0,1]$.
In \cite[Thm~5.15]{XB17} a constant $c \in {]0,\infty[} \cup 
\{\infty\}$ is given explicitly such that
\begin{align}\label{eq:xuberger:c}
\lim_{p \to \infty} 2^p \cdot \rbit(\mu,p) = c.
\end{align}
In particular, $\rbit(\mu,p) \succeq 2^{-p}$, and this lower 
bound is sharp if and only if $c< \infty$.
As an elementary example we have $c = (2 \cdot \sqrt{3})^{-1}$ for
$\mu$ being the uniform distribution on the unit interval.

Next, assume that all moments of $\mu$ are finite.
Then we have
\begin{align}\label{eq:xuberger:upper}
\rbit(\mu,p) \preceq (2^p)^{-1/2 + \eps} 
\end{align}
for all $\eps > 0$, see \cite[Thm.~5.20]{XB17}.
\end{rem}

\begin{rem}\label{r2.3}
Random bit approximation is
closely related to quantization, which has been studied
intensively for finite-dimensional and for infinite-dimensional
spaces $V$. More precisely, let
\[
\F(V,p) = \{ \nu \in \M(V) \colon |\supp(\nu)| \leq 2^p \}
\]
denote the set of all probability measures on $V$ with support of
size at most $2^p$. The quantization numbers 
\[
\quant(\mu,p) = \inf\{ d(\mu, \nu) \colon \nu \in \F(V,p)\}
\]
immediately yield lower bounds for $\rbit(\mu,p)$, i.e.,
\[
\rbit(\mu,p) \geq \quant(\mu,p)
\]
for every $\mu \in \M(V)$ and every $p \in \N$. 
A partial list of references on quantization of probability
measures includes the monograph \cite{GL00} and the survey
\cite{D09} as well as
\cite{CDMGR09,D03,D07a,D07,DFMS03,DS05,LP02,LP04,LP06}.

The results from \cite{DSS13}, which deals with
quantization on $V=\R^d$ by
means of empirical measures, immediately yield upper bounds for
$\rbit(\mu,p)$. In particular,
if $d \geq 5$ and if $\mu$ has a finite moment of any
order greater than $2/(1-2/d)$, then $\rbit (\mu,p) \preceq
2^{-p/d}$, see \cite[Thm.\ 1]{DSS13}. This upper bound is sharp
in many cases, since $\quant(\mu,p) \succeq 2^{-p/d}$ under mild 
assumptions on $\mu$ for every $d \in \N$, see \cite{D09,GL00} for 
details.

We stress that the strong asymptotics of $\quant(\mu,p)$
is studied most of the time in the literature, while we only
consider the weak asymptotics of $\rbit(\mu,p)$. Observe that 
we lose the control about asymptotic constants anyway in the analysis 
of the random bit quadrature problem.
\end{rem}

\begin{rem}\label{r2.4}
Let $(W,\|\cdot\|_W)$ denote another separable Banach space.
For the proof of upper bounds for $\rbit(\mu,p)$ we may use the 
following simple observation. Let $f\colon W\to V$ be measurable
and $\mu \in \V(W,p)$, then $\mu_f\in \V(V,p)$.
\end{rem}

\subsection{Approximation of the Standard Normal Distribution}\label{sec2.2}

We first fix some notations. For $p \in \N$ let
\begin{align*}
D^{(p)} &= \Bigl\{\sum_{i=1}^p b_i \cdot 2^{-i} + 2^{-(p+1)} \colon 
\text{$b_i \in \{0,1\}$ for $i=1,\dots,p$}\Bigr\}\\
&= \{ k \cdot 2^{-p} - 2^{-(p+1)} \colon k = 1,\dots, 2^p\}
\end{align*}
denote the set of dyadic numbers from ${[0,1[}$ with $p$ bits, 
shifted by $2^{-(p+1)}$, so that $D^{(p)}$ is symmetric with
respect to $1/2$. 
Furthermore, we define the truncation operator $T^{(p)}$ via
\[
T^{(p)} \colon {[0,1[} \to D^{(p)}, 
x \mapsto \frac{\lfloor 2^p \, x \rfloor}{2^p} + 2^{-(p+1)},
\]
i.e., the application of $T^{(p)}$ means rounding
to a nearest element from $D^{(p)}$.

Let $Y$ be a standard normally distributed random variable and let 
$\Phi$ denote the corresponding distribution function. Observe that
$U=\Phi(Y)$ is uniformly distributed on $[0,1]$, so that $T^{(p)}(U)$
is uniformly distributed on $D^{(p)}$. 
The distribution of
\begin{equation}\label{eq:stnormal-approx}
Y^{(p)} = \Phi^{-1} \circ T^{(p)} \circ \Phi (Y) 
\end{equation}
therefore belongs to $\U(\R,p)$.

\begin{theo}\label{p:approx-stnormal-prop} 
Let $\mu$ denote the standard normal distribution. Then
we have
\begin{align}\label{th1eq1}
\rbit(\mu,p)
\asymp
2^{-p/2} \cdot p^{-1/2}.
\end{align}
Furthermore, let $Y^{(p)}$ as in \eqref{eq:stnormal-approx}. Then
\begin{align}\label{th1eq2}
\left(\E|Y - Y^{(p)}|^2\right)^{1/2}
\asymp
\rbit(\mu,p).
\end{align}
Moreover,
\begin{align}\label{th1eq3}
\E\bigl(Y^{(p)}\bigr) = 0
\end{align}
and
\begin{align}\label{th1eq4}
\sup_{p \in \N} \E\bigl|Y^{(p)}\bigr|^r < \infty
\end{align}
for all $r \geq 1$. 
\end{theo}

\begin{proof}
By definition,
\begin{align}\label{th1eq5}
\rbit(\mu,p) \leq 
\bigl(\E|Y - Y^{(p)}|^2\bigr)^{1/2}.
\end{align}
Hence we show that
\begin{equation}\label{eq:upper}
\bigl(\E|Y - Y^{(p)}|^2\bigr)^{1/2} \preceq 2^{-p/2} \cdot p^{-1/2}.
\end{equation}
Let $z_k = k \cdot 2^{-p}$ for $k = 2^{p-1},\dots, 2^p$, and let 
$\varphi$ denote the density of the standard normal distribution. We have 
\[
\E|Y - Y^{(p)}|^2
=
\E\bigl|\Phi^{-1}(U) - \Phi^{-1}\circ T^{(p)}(U)\bigr|^2 
=
2 \cdot  
\sum_{k=2^{p-1}+1}^{2^p} A_k,
\]
where
\[
A_k
=
\int_{\left[z_{k-1},z_{k} \right[}  
\bigl| \Phi^{-1} (u) - x_k \bigr|^2 \, \mathrm{d}u
\]
with $x_k = \Phi^{-1}(z_k - 2^{-(p+1)} )$.

Consider the case $2^{p-1}+1 \leq k \leq 2^p-2$.
Observe that $\varphi \circ \Phi^{-1}$ is monotonically
decreasing on
$\left[1/2,1\right[$.
Using 
\[
\bigl| \Phi^{-1} (u) - x_k  \bigr|
\leq
\frac{|u-(z_k-2^{-(p+1)})|}{\varphi(\Phi^{-1}(z_{k}))} 
\]
for $u \in \left[z_{k-1},z_k\right[$, we obtain
\[
A_k
\leq 
\frac{2^{-3p}}{12 \, \varphi^2(\Phi^{-1}(z_k))}.
\]
Consequently, 
\begin{align*}
2^p\, p \cdot \sum_{k=2^{p-1}+1}^{2^p-2} A_k
&\leq
\frac{2^{-2p}\, p}{12} \cdot 
\sum_{k=2^{p-1}+1}^{2^p-2} \frac{1}{\varphi^2(\Phi^{-1}(z_k))} 
\leq
\frac{2^{-p}\, p}{12} \cdot 
\int_{1/2}^{1-2^{-p}} \frac{1}{\varphi^2(\Phi^{-1}(u))}\, \mathrm{d}u \\
&=
\frac{2^{-p}\, p \cdot \sqrt{2 \pi}}{12} \cdot 
h(\Phi^{-1}(1-2^{-p})), 
\end{align*}
where
\begin{equation}\label{eq:h}
h(a) = 
\int_{0}^{a} \exp(x^2/2) \, \mathrm{d}x
\end{equation}
for $a > 0$. By Lemma~\ref{a:l:h} from the Appendix we obtain
\[
\limsup_{p \to \infty} 2^p\, p \cdot \sum_{k=2^{p-1}+1}^{2^p-2} A_k
\leq
\frac{1}{12 \, \ln 4}.
\]

For every $2^{p-1}+1 \leq k \leq 2^p$,
\[
A_k \leq 2 \cdot \int_{\left[z_{k-1}+2^{-(p+1)},z_{k}\right[}
\bigl| \Phi^{-1} (u) - x_k\bigr|^2 \, \mathrm{d}u
= 2 \cdot
\int_{\left[x_k,\Phi^{-1}(z_k) \right[} 
(x - x_k)^2 \, \varphi(x) \, \mathrm{d}x.
\]
Furthermore, this upper bound for $A_k$ is monotonically increasing
in $k$, since $\Phi^{-1}$ is convex on $\left[1/2,1\right[$. 
Therefore we consider the case $k=2^p$. Put
\begin{equation}\label{eq:g}
g(a)
= 
\int_{\left[a,\infty \right[} (x - a)^2 \, \varphi(x) \, \mathrm{d} x
\end{equation}
for $a > 0$.
By Lemma~\ref{a:l:g} from the Appendix we obtain
\[
\limsup_{p \to \infty} 2^p\, p 
\cdot \E | Y - Y^{(p)}|^2 
\leq
\frac{49}{6 \, \ln 4},
\]
which completes the proof of the upper bound \eqref{eq:upper}.

Next we show
\begin{equation}\label{eq:lower}
\rbit(\mu,p) \succeq 2^{-p/2} \cdot p^{-1/2}.
\end{equation}
To this end we consider a random variable
$\widehat{Y}^{(p)}$ with distribution in 
$\V(\mathbb{R},p)$
and defined on the same space as $Y$.
Let $\widehat{x}$ denote the  
essential supremum of $\widehat{Y}^{(p)}$.
We are going to consider two cases, at first we assume
\begin{align*}
\widehat{x} \leq \Phi^{-1}(1 - 2^{-(p+1)}).
\end{align*}
Due to the monotonicity of $\Phi^{-1}$ we have
\begin{align*}
\E \bigl(Y - \widehat{Y}^{(p)}\bigr)^2
&\geq
\E \Bigl(\bigl(\Phi^{-1}(U) - \widehat{Y}^{(p)}\bigr)^2 \cdot 
1_{\{U \geq 1 - 2^{-(p+1)}\}}\Bigr)\\
&\geq \E\Bigl(\bigl(\Phi^{-1}(U) - \Phi^{-1}(1 - 2^{-(p+1)})\bigr)^2 
\cdot 1_{\{U \geq 1 - 2^{-(p+1)}\}}\Bigr)\\
&=
g\bigl(\Phi^{-1}(1-2^{-(p+1)})\bigr).
\end{align*}
Lemma~\ref{a:l:g} from the Appendix yields
$g\bigl(\Phi^{-1}(1-2^{-(p+1)})\bigr) \succeq 2^{-p} \cdot p^{-1}$.
For the second case, i.e.,
\begin{align*}
\widehat{x} > \Phi^{-1}(1 - 2^{-(p+1)})
\end{align*}
we use Lemma~\ref{a:Phi-mean-val} from the Appendix with 
$a = 1 - 3 \cdot 2^{-(p+2)}$ and $b = 1 - 2^{-(p+1)}$ to conclude that
\[
\Phi^{-1}(b) - \Phi^{-1}(a) \succeq p^{-1/2}.
\]
Together with the monotonicity of $\Phi^{-1}$ this leads to
\begin{align*}
\E \bigl(Y - \widehat{Y}^{(p)}\bigr)^2
&\geq
\E \Bigl(\bigl(\Phi^{-1}(U) - \widehat{x}\bigr)^2 \cdot 
1_{\{\widehat{Y}^{(p)} = \widehat{x}\}} \cdot 1_{\{U \leq a\}}\Bigr)\\
&\geq
\E \Bigl(\bigl(\Phi^{-1}(a) - \Phi^{-1}(b)\bigr)^2 \cdot 
1_{\{\widehat{Y}^{(p)} = \widehat{x}\}} \cdot 1_{\{U \leq a\}}\Bigr)\\
&\succeq
p^{-1} \cdot P\bigl(\{\widehat{Y}^{(p)} = \widehat{x}\} \cap \{U \leq
a\}\bigr)\\
&\geq
p^{-1} \cdot \bigl(P(\{\widehat{Y}^{(p)} = \widehat{x}\}) - 
P(\{U > a\})\bigr)\\
&
\geq p^{-1} \cdot \bigl(2^{-p} - P(\{U > a\})\bigr) \\
&=
p^{-1} \cdot 2^{-(p+2)}.
\end{align*}
This completes the proof of \eqref{eq:lower}.
Combing \eqref{th1eq5}, \eqref{eq:upper}, and \eqref{eq:lower}
yields
\eqref{th1eq1}
and
\eqref{th1eq2}.

For the proof of \eqref{th1eq3} we observe that $T^{(p)}(U)$ is 
uniformly distributed on $D^{(p)}$. Since 
$\Phi^{-1}(1-x) = - \Phi^{-1}(x)$ for 
$x \in ]0,1[$, we get \eqref{th1eq3} for symmetry reasons.

Observe that
$x \mapsto (\Phi^{-1}(x))^r$ is a convex function
on ${[1/2,1[}$ for $r \geq 1$ with integral equal to
$\tfrac{1}{2} \E |Y|^r$, and applying a midpoint rule
to this function we get $\tfrac{1}{2} \E|Y^{(p)}|^r$.
Hence we have
\[
\sup_{p \in \N} \E\bigl|Y^{(p)}\bigr|^r \leq \E|Y|^r < \infty,
\]
which implies \eqref{th1eq4}.
\end{proof}

\begin{rem}
We compare Theorem~\ref{p:approx-stnormal-prop} with the results
from \cite{XB17}, as discussed in Remark~\ref{xuberger:bounds},
for the standard normal distribution $\mu$.

Note that the measure corresponding to $\Phi^{-1}$ is absolutely 
continuous with respect to the Lebesgue measure on $[0,1]$. 
Theorem~\ref{p:approx-stnormal-prop} implies that 
$c = \infty$ in \eqref{eq:xuberger:c}. Moreover, 
the order of convergence of $\rbit(\mu,p)$ is only
slightly better than the upper bound \eqref{eq:xuberger:upper},
which holds for every $\mu \in \M(\R)$ with finite moments
of any order.

The optimal selection of support points $x_k^*$
is given by the local averages
of $\Phi^{-1}$ based on a uniform partition of $[0,1]$, see
\eqref{g55} with $\Psi^{-1} = \Phi^{-1}$. In 
Theorem~\ref{p:approx-stnormal-prop} we consider a slightly simpler
construction, which still yields the same order of convergence
of the Wasserstein distance,
as we employ the values $x_k$ of $\Phi^{-1}$ at the
midpoints for this partition. Both of these point sets
are symmetric with respect to $1/2$, and
$x_k < x_k^* < x_{k+1}$ for $k=2^{p-1}+1,\dots,2^{p}-1$.
\end{rem}

\begin{rem}\label{r1}
We compare Theorem \ref{p:approx-stnormal-prop} 
with known results
for the quantization problem, see, e.g., \cite[Thm. 6.2]{GL00}.
First of all,
\[
\quant(\mu, p)
\asymp
2^{-p}
\]
for the quantization of the standard normal distribution
$\mu$, so that the quantization error converges to zero much faster
than the corresponding quantity for random bit approximation.
On the other hand, for the uniform distribution on $[0,1]$
both quantities are of the same order $2^{-p}$.
\end{rem}

\subsection{Approximation of the Distribution of a Brownian 
Bridge}\label{s3}

Let $(s_i)_{i \in \mathbb{N}}$ be the sequence of Schauder 
functions given by
\[
\phantom{\quad\qquad t \in [0,1],}
s_i(t) = \int_0^t h_i(u) \, \mathrm{d} u, 
\quad\qquad t \in [0,1],
\]
with
\[
h_{2^m+k-1} = 2^{m/2} \cdot
\left( 1_{I_{2^m+k}} - 1_{J_{2^m+k}} \right)
\]
for $m \in \N_0$ and $k=1,\dots,2^m$,
where
\[
I_{2^m+k}
=
\left[
(k-1) / 2^m,
(k-1/2) / 2^m
\right[
\]
and
\[
J_{2^m+k}
=
\left[
(k-1/2) / 2^m,
k / 2^m
\right[.
\]

Let $B$ denote a standard Brownian bridge on $[0,1]$, which is
henceforth considered as a centered Gaussian random element that takes
values in $H=L_2=L_2([0,1])$.
The L\'evy-Ciesielski (or Brownian bridge) representation 
of $B$ states that 
\[
B = \sum_{i=1}^\infty Y_i \cdot s_i
\]
with convergence, e.g., in mean-square with respect to the $L_2$-norm.
Here $Y_1,Y_2,\dots$ is an independent sequence of standard
normally distributed random variables.

We define
\[
B^{(\ell)}
= \sum_{i=1}^{2^\ell-1} Y_i \cdot s_i
\]
for $\ell \in \N$, i.e., $B^{(\ell)}$ is the piecewise linear interpolation of $B$
at the points $k \cdot 2^{-\ell}$ with $k=0,\dots,2^\ell$.
The following result is well known,
see, e.g., \cite[Sec.\ II.3]{R00} for references and remarks.

\begin{lemma}\label{p:BB-approx-rate}
We have
\[
\left(\E\bigl\|B - B^{(\ell)}\bigr\|_{L_2}^2\right)^{1/2} 
\asymp 2^{-\ell/2}.
\]
\end{lemma}

For $\ell \in \mathbb{N}$ we consider a vector 
\[
\bp = (p_1,\ldots,p_{2^\ell-1}) \in 
\mathbb{N}^{2^\ell-1}
\]
of bit numbers. 
We define
\[
B^{(\ell,\bp)}
= \sum_{i=1}^{2^\ell-1} Y_i^{(p_i)} \cdot s_i,
\]
where $Y_i^{(p_i)}$ 
is the approximation of $Y_i$ according to 
\eqref{eq:stnormal-approx}.
This approach, which is appropriate for the construction
of multilevel algorithms, see Section \ref{sec3}, 
has been suggested in \cite[p.~320]{G15}.
Note that the distribution of $B^{(\ell,\bp)}$ belongs to
$\U(L_2,|\bp|)$ with
\[
|\bp| = \sum_{i=1}^{2^{\ell}-1} p_i.
\]

\begin{lemma}\label{p:BB-approx-vs-Bp}
We have
\[
\left(\E\bigl\|B^{(\ell)} - B^{(\ell,\bp)}\bigr\|_{L_2}^2\right)^{1/2}
\asymp 
\biggl(\sum_{i=1}^{2^\ell-1} 
2^{-p_i} / p_i \cdot i^{-2}
\biggr)^{1/2}
\]
uniformly in $\ell \in \N$ and $\bp \in \N^{2^\ell-1}$.
\end{lemma}

\begin{proof}
Let $\widehat{Y}_i = Y_i^{(p_i)}$.
Use Theorem~\ref{p:approx-stnormal-prop} and $\|s_i\|_{L_2}^2 \asymp 
i^{-2}$ to obtain
\begin{align*}
&\E \bigl\|B^{(\ell,\bp)} - B^{(\ell)}\bigl\|_{L_2}^2 
=
\int_0^1 \E \biggl(\sum_{i=1}^{2^\ell-1}
\bigl(\widehat{Y}_i - Y_i\bigr) \cdot s_i(t)\biggr)^2 \, \mathrm{d}t\\
&\qquad=
\int_0^1 \Var \biggl(\sum_{i=1}^{2^\ell-1}
\bigl(\widehat{Y}_i - Y_i\bigr) \cdot s_i(t)\biggr) \, \mathrm{d}t
=
\int_0^1 \sum_{i=1}^{2^\ell-1}\Var
\bigl(\bigl(\widehat{Y}_i - Y_i\bigr) \cdot s_i(t)\bigr) \, \mathrm{d}t\\
&\qquad=
\sum_{i=1}^{2^\ell-1} \E 
\bigl(\widehat{Y}_i - Y_i\bigr)^2 \cdot \int_0^1 s_i^2(t) \, \mathrm{d}t
\asymp
\sum_{i=1}^{2^\ell-1} 2^{-p_{i}} / p_i \cdot i^{-2}.
\qedhere
\end{align*}
\end{proof}

\begin{theo}\label{p:Bp-rate}
Let $\mu$ be the distribution of a standard Brownian bridge $B$ on $L_2$. 
Then we have
\[
\rbit(\mu,p) \asymp \quant(\mu,p) \asymp p^{-1/2}.
\]
Define $\bp(\ell) \in \N^{2^\ell-1}$ for $\ell \in \N$ by 
\begin{equation}\label{eq:prec-setup}
\phantom{\qquad\quad i =1,\ldots,2^\ell-1.}
p_i(\ell) = 2 \cdot(\ell - \lfloor\log_2 i\rfloor),
\qquad\quad i =1,\ldots,2^\ell-1.
\end{equation}
Then we have
\[
\left(\E \bigl\|B-B^{(\ell,\bp(\ell))}\bigr\|_{L_2}^2\right)^{1/2}
\asymp
\rbit(\mu,|\bp(\ell)|) 
\]
and
\[
|\bp(\ell)| = 2^{\ell+2} -2\ell - 4 \asymp 2^{\ell}.
\]
\end{theo}

\begin{proof}
We write $\bp$ and $p_i$ instead of $\bp(\ell)$ and $p_i(\ell)$,
respectively, to simplify the notation.
By definition, 
\[
\rbit(\mu,|\bp|) \leq 
\left(\E \bigl\|B-B^{(\ell,\bp)}\bigr\|_{L_2}^2\right)^{1/2}. 
\]
Hence we show that
\begin{equation}\label{g98}
\left(\E \bigl\|B-B^{(\ell,\bp)}\bigr\|_{L_2}^2\right)^{1/2}
\preceq |\bp|^{-1/2}.
\end{equation}
Since
\begin{align}\label{eq:precision:necessary}
\sum_{i=1}^{2^\ell-1} 
2^{-p_{i}} / p_i \cdot i^{-2}
\leq
2^{-\ell}
\end{align}
for the specific choice of the bit numbers $p_i$,
Lemmata \ref{p:BB-approx-rate} and \ref{p:BB-approx-vs-Bp}
yield 
\[
\left(\E \bigl\|B-B^{(\ell,\bp)}\bigr\|_{L_2}^2\right)^{1/2} \preceq 
2^{-\ell/2}. 
\]
The explicit formula for $|\bp|$ is easily verified by induction,
and this completes the proof of the asymptotic upper bound \eqref{g98}.
On the other hand,
\[
\quant(\mu,p) \asymp p^{-1/2},
\]
see \cite[p.\ 527]{LP02} and \cite{D03}.
\end{proof}

\begin{rem}
Observe that the Schauder function $s_i$ has a support
of size $2^{- \lfloor \log_2 i \rfloor}$. Therefore
$B^{(\ell,\bp(\ell))}$ involves all Schauder functions with support size
between $1$ and $2^{-(\ell-1)}$, and the number of random bits that
is associated to $s_i$ according to \eqref{eq:prec-setup} only
depends on the size of its support. This number varies linearly
between $2 \ell$ for $s_1$ and $2$ for $s_i$ with 
$i=2^{\ell-1},\ldots, 2^{\ell}-1$.

In our construction the total number $|\bp(\ell)|$ of bits
coincides, up to a multiplicative constant,
with the numbers of terms in $B^{(\ell,\bp(\ell))}$ and in $B^{(\ell)}$. 
The partial sum $B^{(\ell)}$ formally corresponds to 
$p_i(\ell) = \infty$ for
$i=1,\dots,2^\ell-1$, but still the errors of $B^{(\ell)}$
and $B^{\ell,\bp(\ell))}$ are of the same order
$2^{-\ell/2}$, see Lemma \ref{p:BB-approx-rate} and
Theorem \ref{p:Bp-rate}.
\end{rem}

\begin{rem}
The bit numbers given by \eqref{eq:prec-setup} 
depend on $i$ and they approximately 
minimize $|\bp|$, subject to the constraint 
\eqref{eq:precision:necessary}.

For constant bit numbers 
\begin{equation}\label{eq:prec-constant}
p=
p_1 = \ldots = p_{2^\ell-1}
\end{equation}
the following holds true.
For \eqref{eq:precision:necessary} to hold true we must have 
$2^{-p}/p\leq 2^{-\ell}$,
and together with $p\leq 2^{p}$ this yields 
$p \geq \ell/2$. On the other hand $p = 2 \cdot \ell$ 
implies \eqref{eq:precision:necessary}. 
Therefore the minimum of $|\bp|$, subject to the constraints
\eqref{eq:precision:necessary} and \eqref{eq:prec-constant} is only 
of the order $2^\ell \cdot \ell$.
\end{rem} 

\subsection{Approximation of Gaussian Measures}\label{randombitappsec}

In this section we consider a centered Gaussian random element $X$ 
that takes values in a separable Hilbert space $(H,\|\cdot\|_H)$ 
and has an infinite-dimensional support.
The Karhunen-Lo\`eve expansion of $X$ may be written as
\begin{align}\label{series}
X = \sum_{i=1}^\infty \lambda_i^{1/2} \cdot Y_i \cdot e_i
\end{align}
with convergence, e.g., in mean-square with respect to the norm of $H$.
Here $(e_i)_{i \in \N}$ is an orthonormal system in $H$ and 
$(\lambda_i)_{i\in\N}$ is a non-increasing and summable 
sequence of strictly positive numbers,
and
$Y_1,Y_2,\dots$ is an independent sequence of standard normally
distributed random variables.
We assume that
\begin{equation}\label{n-g2}
\lim_{i\to\infty}
\lambda_i \cdot i^{\beta} \cdot (\ln(i))^{\alpha}
\in \left]0,\infty\right[,
\end{equation}
where $\beta > 1$ and $\alpha \in \R$.
The asymptotic behavior of the variances $\lambda_i$ of the random
coefficients of $X$ is known in many cases.
For instance, 
\[
\beta = 2h +1
\]
and 
\[
\alpha= -(d-1)\cdot \beta
\]
for a fractional Brownian sheet $X$ on $[0,1]^d$ with Hurst parameter 
$h \in \left]0,1\right[$ and $H=L_2([0,1]^d)$, see, e.g., 
\cite[p.~1586, p.~1588]{LP04}. In particular,
$\beta=2$ and $\alpha=0$ for a Brownian motion, as well as for a 
Brownian bridge. 

The analysis from the previous section extends to the
case of Gaussian random elements in a straight-forward way.
In contrast to the L\'evy-Ciesielski representation the
Karhunen-Lo\`eve expansion may naturally be truncated after any
number of terms.
For $m \in \mathbb{N}$ we consider a vector 
\[
\bp = (p_1,\ldots,p_m) \in \mathbb{N}^{m}
\]
of bit numbers. We define
\begin{align}\label{gaussian}
X^{(m,\bp)}
= \sum_{i=1}^m \lambda_i^{1/2} \cdot Y^{(p_i)}_i \cdot e_i,
\end{align}
where $Y_i^{(p_i)}$ 
is the approximation of $Y_i$ according to 
\eqref{eq:stnormal-approx}.
Note that the distribution of $X^{(m,\bp)}$ belongs to
$\U(H,|\bp|)$ with
\[
|\bp| = \sum_{i=1}^{m} p_i.
\]

\begin{theo}\label{n-t1}
Let $\mu$ denote the distribution of the Gaussian random element
$X$ on $H$, and assume that \eqref{n-g2} is 
satisfied.
Then we have
\begin{align*}
\rbit(\mu,p)
\asymp
\quant(\mu,p) \asymp 
p^{-(\beta-1)/2} \cdot (\ln(p+1))^{-\alpha/2}.
\end{align*}
Define $\bp(m) \in \N^m$ for $m \in \N$ by
\[
\phantom{\qquad\quad i=1,\dots,m,}
p_i(m) = \lceil \max (\tilde{p}_i(m),1) \rceil,
\qquad\quad i=1,\dots,m,
\]
where
\[
\tilde{p}_i(m) =
\beta \cdot \log_2 (m/i) 
+\max(\alpha,0)\cdot \log_2 (\log_2 (m+1) /\log_2(i+1) ).
\]
Then we have
\[
\left(\E \bigl\|X-X^{(m,\bp(m))}\bigr\|_{H}^2\right)^{1/2}
\asymp
\rbit(\mu,|\bp(m)|)
\]
and
\[
|\bp(m)| \asymp m.
\]
\end{theo}

\begin{proof}
We write $\bp$, $p_i$, and $\tilde{p}_i$ instead of $\bp(m)$,
$p_i(m)$, and $\tilde{p}_i(m)$, respectively, to simplify the notation.
Note that
\[
\int_0^1 \frac{1}{x \cdot t+1}\, dt = \frac{\ln(x+1)}{x}
\]
for $x>0$.
It follows that 
$x \mapsto x/\log_2(x+1)$
is strictly increasing on $\left]0,\infty\right[$.
Therefore $m/i \geq \log_2(m+1) / \log_2(i+1)$ and 
\[
p_i \leq 1 + \tilde{p}_i \preceq 1 + \ln(m/i)
\]
uniformly in $m \in \N$ and $i=1,\dots,m$.
Since
\[
\sum_{i=2}^m \ln(m/i) \leq
\int_1^m\ln(m/x)\,\mathrm dx
=m-\ln(m) -1,
\]
we obtain
\begin{align*}
|\bp|
\preceq
m+ \sum_{i=1}^m \ln(m/i)
\leq
2m-1,
\end{align*}
while $|\bp|\geq m$ trivially holds true. We conclude that
$|\bp| \asymp m$, as claimed.

By definition,
\[
\rbit(\mu,|\bp|) \leq 
\left(\E \bigl\|X-X^{(m,\bp)}\bigr\|_{H}^2\right)^{1/2}.
\] 
Hence we show that
\begin{equation}\label{n-g3}
\left(\E \bigl\|X-X^{(m,\bp)}\bigr\|_{H}^2\right)^{1/2}
\preceq
m^{-(\beta-1)/2} \cdot (\ln (m+1))^{-\alpha/2}.
\end{equation}
First of all,
\[
X^{(m)} = \sum_{i=1}^m \lambda_i^{1/2} \cdot Y_i \cdot e_i
\]
with $m \in \N$ satisfies 
\[
\E\bigl\|X - X^{(m)}\bigr\|_{H}^2
\asymp m^{-(\beta-1)} \cdot (\ln (m+1))^{-\alpha},
\]
see \eqref{n-g2}.
Furthermore, Theorem~\ref{p:approx-stnormal-prop} yields 
\[
\E \bigl\|X^{(m)} - X^{(m,\bp)}\bigl\|_{H}^2 
=
\sum_{i=1}^{m} \E 
\bigl(Y_i-Y_i^{(p_i)}\bigr)^2 \cdot \lambda_i 
\asymp
\sum_{i=1}^{m} 
2^{-p_{i}} / p_i \cdot i^{-\beta} \cdot (\ln(i+1))^{-\alpha}
\]
uniformly in $m \in \N$ and $\bp \in \N^{m}$.
For the specific choice of bit numbers $p_i$ we obtain
\[
2^{-p_{i}} \cdot i^{-\beta} \cdot (\ln(i+1))^{-\alpha}
\preceq m^{-\beta} \cdot (\ln(m+1))^{-\alpha}
\]
uniformly in $m \in \N$ and $i=1,\dots,m$. Since $p_i \geq 1$, we
conclude that
\[
\E \bigl\|X^{(m)} - X^{(m,\bp)}\bigl\|_{H}^2 
\preceq 
m^{-(\beta-1)} \cdot (\ln(m+1))^{-\alpha},
\]
which completes the proof of \eqref{n-g3}.

On the other hand,
\[
\quant(\mu,p)
\asymp
p^{-(\beta-1)/2} \cdot (\ln (p+1))^{-\alpha/2},
\]
see, e.g., \cite[p.\ 1581]{LP04}
\end{proof}

\subsection{Approximation of the Distribution of a Scalar SDE}

We consider a scalar autonomous SDE
\begin{align*}
\mathrm{d} X(t)
&=
a(X(t))\,\mathrm{d}t + b(X(t))\,\mathrm{d}W(t), \quad t \in [0,1],\\
X(0)
&=
x_0
\end{align*}
with a deterministic initial value $x_0 \in \mathbb{R}$ and a
scalar Brownian motion $W$. Both, the drift coefficient $a\colon\R\to\R$
and the diffusion coefficient $b\colon\R\to\R$
are assumed to be differentiable with bounded and Lipschitz continuous 
derivatives. This yields, in particular, 
\begin{equation}\label{g5}
\E \sup_{t \in [0,1]} |X(t)|^2 < \infty.
\end{equation}
Furthermore, we assume that $b(x_0) \neq 0$ in order to
exclude the case of a deterministic equation.

At first, we consider the random bit approximation of marginal
distributions of $X$.
To this end we consider the Milstein scheme based on the equidistant 
points
\[
\phantom{\qquad\quad k = 0,\ldots,m,}
t_k = t_{k,m} = k/m, 
\qquad\quad k = 0,\ldots,m,
\]
where $m \in \N$.
In terms of the normalized increments
\[
Y_k = Y_{k,m} = m^{1/2} \cdot \left(W(t_k) - W(t_{k-1})\right)
\]
the scheme reads as
\begin{align*}
X_m(t_0)
&=
x_0,\\
X_m(t_k)
&=
X_m(t_{k-1}) + a\bigl(X_m(t_{k-1})\bigr) \cdot m^{-1} + 
b\bigl(X_m(t_{k-1})\bigr) \cdot m^{-1/2} \cdot Y_k\\
&\hspace{1cm}+ \tfrac{1}{2}\cdot (b \cdot b') 
\bigl(X_m(t_{k-1})\bigr) \cdot m^{-1} \cdot \bigl(Y_k^2 -1\bigr),
\end{align*}
where $k = 1,\ldots,m$. The following result is well known,
see, e.g., \cite[Thm.~1.2.4]{MT04}.

\begin{lemma}\label{l1}
We have
\begin{align*}
\left(
\E \left( \max_{k =1,\ldots,m} |X(t_k) - X_m(t_k)|^2 
\right) \right)^{1/2}
\preceq m^{-1}.
\end{align*}
\end{lemma}

Let $q \in \N$.
The approximation 
\[
Y_k^{(q)} = Y_{k,m}^{(q)} = \Phi^{-1} \circ T^{(q)} \circ \Phi (Y_k),
\]
of the normalized increments,
cf.~\eqref{eq:stnormal-approx}, leads to the random bit Milstein scheme
\begin{align*}
X_m^{(q)}(t_0)
&=
x_0,\\
X_m^{(q)}(t_k)
&=
X_m^{(q)}(t_{k-1}) + a\left(X_m^{(q)}(t_{k-1})\right) \cdot m^{-1} + 
b\left(X_m^{(q)}(t_{k-1})\right) \cdot m^{-1/2} \cdot Y_k^{(q)}\\
&\qquad + \tfrac{1}{2} \cdot (b\cdot b')\left(X_m^{(q)}(t_{k-1})\right) 
\cdot m^{-1} \cdot \big(\big(Y_k^{(q)}\big)^2 -1\big),
\end{align*}
where $k = 1,\ldots,m$.

We are going to employ results from \cite{MGR13}, which 
deals with the quantization problem. In the latter setting approximations 
$\widetilde{Y}^{(q)}_k$ to $Y_k$ with distributions in $\F(\R,q)$ and 
error of order $2^{-q}$ are available, see Remark \ref{r1}. 
However, the method of proof for Lemma 3 from \cite{MGR13} 
is immediately applicable in the present setting of random bit
approximation, where we
rely on Theorem~\ref{p:approx-stnormal-prop}.

\begin{lemma}[Cf.\ {\cite[Lemma 3]{MGR13}}]\label{l2}
We have
\begin{align*}
\left(
\E \left(\max_{k =1,\ldots,m} |X_m(t_k) - X_m^{(q)}(t_k)|^2\right) 
\right)^{1/2}
\preceq m^{-1} + 2^{-q/2} \cdot q^{-1/2}
\end{align*}
uniformly in $m,q \in \N$.
\end{lemma}

\begin{rem}
Let $\nu^{(q)}_m$ denote the joint distribution of
$X^{(q)}_m(t_1),\dots,X^{(q)}_m(t_m)$ and let $\nu$ denote
the corresponding marginal distribution of $X$.
Consider the supremum norm on $V=\R^m$.
The joint distribution of $Y^{(q)}_1,\dots,Y^{(q)}_m$ belongs to
$\U(\R^m,m q)$,
and therefore $\nu^{(q)}_m \in \V(\R^m,m q)$.
Lemmata \ref{l1} and \ref{l2} yield
\begin{equation}\label{g2}
\rbit(\nu,mq) \leq \left(
\E \left(\max_{k=1,\dots,m} |X(t_k) - X^{(q)}_m(t_k)|^2 \right)
\right)^{1/2}
\preceq m^{-1} + 2^{-q/2} \cdot q^{-1/2}.
\end{equation}
\end{rem}

Now we turn to the random bit approximation of the distribution of 
$X$ on the space $L_2 = L_2([0,1])$.
We employ a piecewise linear interpolation together with a
local refinement of the Milstein approximation
on each of the subintervals $[t_{k-1},t_k]$. To this end we consider
the Brownian bridges
\[
\phantom{\quad\quad t \in [0,1],}
B_k(t) =  B_{k,m}(t) =
m^{1/2} \cdot \left(W(t_{k-1}+t/m) - W(t_{k-1}) \right)
- t \cdot Y_k, 
\quad\quad t \in [0,1],
\]
and we define
\begin{align*}
\overbar{X}_m(t) &= 
(t - t_{k-1}) \cdot m \cdot X_m(t_k) + (t_{k} - t) \cdot
m \cdot X_m(t_{k-1})\\
&\hspace{1cm} + 
b\bigl(X_m(t_{k-1})\bigr) \cdot m^{-1/2} \cdot B_k((t-t_{k-1})\cdot m),
\end{align*}
where $t \in {[t_{k-1},t_k]}$ and $k=1,\ldots,m$. 

\begin{lemma}[{\cite[Lemma 4]{MGR13}}]\label{l:X-vs-overbarX}
We have
\begin{align*}
\sup_{t \in [0,1]} 
\left( \E |X(t) - \overbar{X}_m(t)|^2 \right)^{1/2} \preceq m^{-1}.
\end{align*}
\end{lemma}

Finally we choose $\bp(\ell)$ according to \eqref{eq:prec-setup}, 
and we define $X_m^{(q,\ell)}$ analogously to 
$\overbar{X}_m$, replacing $X_m$ by $X_m^{(q)}$ and
$B_k$ by $B_k^{(\ell,\bp(\ell))}=B_{k,m}^{(\ell,\bp(\ell))}$. This leads to 
\begin{align*}
X_m^{(q,\ell)}(t) &= 
(t - t_{k-1}) \cdot m \cdot X^{(q)}_m(t_k) + (t_{k} - t) \cdot
m \cdot X^{(q)}_m(t_{k-1})\\
&\hspace{1cm} + 
b\bigl(X^{(q)}_m(t_{k-1})\bigr) \cdot m^{-1/2} \cdot 
B^{(\ell,\bp(\ell))}_k((t-t_{k-1})\cdot m),
\end{align*}
where $t \in {[t_{k-1},t_k]}$ and $k=1,\ldots,m$. 
The distribution of $B_k^{(\ell,\bp(\ell))}$ belongs to 
$\U (L_2, |\bp(\ell)|)$,  and $|\bp(\ell)| 
= 2^{\ell+2} - 2\ell -4$, see Theorem \ref{p:Bp-rate}.
Therefore
the distribution of $X_m^{(q,\ell)}$ belongs 
to $\V(L_2, c(m,q,\ell))$,
where
\[
c (m,q,\ell) = m\cdot(q + 2^{\ell+2} -2\ell-4).
\]

\begin{lemma}\label{p:main}
We have
\begin{align*}
\left(\mathrm{E} \big\|X - X_m^{(q,\ell)}\big\|_{L_2}^2\right)^{1/2} 
\preceq m^{-1} + 2^{-q/2} \cdot q^{-1/2} + m^{-1/2} \cdot 2^{-\ell/2}
\end{align*}
uniformly in $m,q,\ell \in \N$.
\end{lemma}

\begin{proof}
We closely follow the proof of \cite[Lemma 5]{MGR13},
and we write $\bp$ instead of $\bp(\ell)$ to simplify the notation.
Due to Lemma~\ref{l:X-vs-overbarX} it suffices to analyze
$\overbar{X}_m - X_m^{(q,\ell)}$.
This difference is split up into 
\begin{align*}
\overbar{X}_m - X_m^{(q,\ell)} = U_1 + U_2 + U_3,
\end{align*}
where
\begin{align*}
U_1(t) = (t - t_{k-1}) \cdot m \cdot 
\bigl(X_m(t_k) - X_m^{(q)}(t_k)\bigr) + 
(t_k - t) \cdot m \cdot \bigl(X_m(t_{k-1}) - X_m^{(q)}(t_{k-1})\bigr),
\end{align*}
as well as
\begin{align*}
U_2(t) = \bigl(b\bigl(X_m(t_{k-1})\bigr) - b\bigl(X_m^{(q)}(t_{k-1})\bigr)
\bigr) \cdot m^{-1/2} \cdot B_k((t-t_{k-1}) \cdot m)
\end{align*}
and
\begin{align*}
U_3(t) = b\bigl(X_m^{(q)}(t_{k-1})\bigr) \cdot m^{-1/2} \cdot 
\bigl(B_k((t-t_{k-1})\cdot m) - 
B_k^{(\ell,\bp)}((t - t_{k-1})\cdot m)\bigr)
\end{align*}
for $t \in {[t_{k-1},t_k]}$.

Put
\[
\Delta = \Delta_m^{(q)} = \max_{k=1,\ldots,m} |X_m(t_k) - X_m^{(q)}(t_k)|,
\]
and observe that
\[
\E \left( \Delta^2 \right) \preceq m^{-2} + 2^{-q} \cdot q^{-1},
\]
see Lemma \ref{l2}. Clearly $|U_1(t)| \leq \Delta$, and therefore
\[
\E \|U_1\|_{L_2}^2 \preceq \E \left( \Delta^2 \right).
\]
The Lipschitz continuity of $b$ yields 
\begin{align*}
|U_2(t)| \preceq
\Delta \cdot m^{-1/2} \cdot B_k((t -t_{k-1}) \cdot m)
\end{align*}
for $t \in [t_{k-1},t_k]$ and $k=1,\ldots,m$. Moreover,
\begin{align*}
\E \|B_k((\cdot -t_{k-1}) \cdot m)\|_{L_2([t_{k-1},t_k])}^2 \asymp m^{-1}.
\end{align*}
We use the independence of $\Delta$ and
$(B_1,\ldots,B_m)$ to conclude that
\[
\E \|U_2\|_{L_2}^2
\preceq
\E  \left(\Delta^2\right) \cdot 
m^{-1} \cdot \sum_{k=1}^m 
\E\|B_k((\cdot - t_{k-1}) \cdot m)\|_{L_2([t_{k-1},t_k])}^2
\asymp
 m^{-1} \cdot \E \left(\Delta^2\right).
\]
Altogether
\begin{align*}
\E \|U_1 + U_2\|_{L_2}^2 
\preceq
m^{-2} + 2^{-q} \cdot q^{-1}.
\end{align*}

It remains to consider the term $U_3$.
{}From \eqref{g5} and \eqref{g2} we get
\[
\sup_{m\in\N,q\in\N}\E \max_{k=1,\dots,m} |X_m^{(q)}(t_k)|^2  < \infty.
\]
Moreover,
\begin{align*}
\E \|B_k((\cdot -t_{k-1}) \cdot m) -
B_k^{(\ell,\bp)}((\cdot -t_{k-1}) \cdot m) 
\|_{L_2([t_{k-1},t_k])}^2 
& \asymp 
m^{-1} \cdot \E \|B - B^{(\ell,\bp)}\|_{L_2}^2 \\
& \asymp m^{-1} \cdot 2^{-\ell},
\end{align*}
see Theorem~\ref{p:Bp-rate}.
Since $b$ satisfies a linear growth condition we get 
\[
\E \|U_3\|_{L_2}^2 \preceq m^{-1} \cdot 2^{-\ell}
\]
from the independence of 
$\max_{k=1,\dots,m} |X_m^{(q)}(t_k)|$ and $(B_1,\ldots,B_m)$.
\end{proof}

\begin{rem}
Suppose that the Euler scheme, instead of the Milstein scheme,
would be employed in the definition of $X_m^{(q,\ell)}$.
Then the first term in the upper bound from Lemma~\ref{p:main} would
change from $m^{-1}$ to $m^{-1/2}$, so that altogether 
\begin{align*}
\left(\mathrm{E} \big\|X - X_m^{(q,\ell)}\big\|_{L_2}^2\right)^{1/2} 
\preceq m^{-1/2} + 2^{-q/2} \cdot q^{-1/2},
\end{align*}
which does not suffice for our purposes.
\end{rem}

\begin{theo}\label{p5}
Let $\mu$ denote the distribution of $X$ on $L_2$. Then we have
\begin{align*}
\rbit(\mu,p)
\asymp
\quant(\mu,p)
\asymp 
p^{-1/2}.
\end{align*}
Furthermore, let
\[
m(\ell)=2^\ell,\qquad q(\ell) = 2 \ell,
\]
and $c(\ell) = c(m(\ell),q(\ell),\ell)$. Then we have
\[
\left(
\mathrm{E} \big\|X - X_{m(\ell)}^{(q(\ell),\ell)}\big\|_{L_2}^2\right)^{1/2} 
\asymp
\rbit(\mu,c(\ell))
\]
and
\[
c(\ell) = 2^{\ell+2} \cdot (2^\ell-1)
\asymp 2^{2 \ell}.
\]
\end{theo}

\begin{proof}
We write $m$, $q$, and $\bp$ instead of $m(\ell)$, $q(\ell)$,
and $\bp(\ell)$, respectively, to simplify the notation.
By definition, 
\[
\rbit(\mu,c(\ell)) \leq 
\left(
\mathrm{E} \big\|X - X_{m}^{(q,\ell)}\big\|_{L_2}^2\right)^{1/2}.
\]
Hence we show that
\begin{equation}\label{g97}
\left(
\mathrm{E} \big\|X - X_{m}^{(q,\ell)}\big\|_{L_2}^2\right)^{1/2}
\preceq c(\ell)^{-1/2}.
\end{equation}
Use Lemma \ref{p:main} to derive
\[
\left(\mathrm{E} \big\|X - X_m^{(q,\ell)}\big\|_{L_2}^2\right)^{1/2} 
\preceq 2^{-\ell}.
\]
The explicit formula for $c(\ell)$ obviously holds true,
and this completes the proof of the asymptotic upper bound \eqref{g97}.

On the other hand,
\[
\quant(\mu,p) \asymp p^{-1/2},
\]
see \cite[Thm.~1.1]{D07}; the same asymptotic result for quantization 
is derived in \cite{CDMGR09,LP06} under stronger assumptions.
\end{proof}

\section{Random Bit Quadrature with respect to Gaussian Measures}\label{sec3}

As in Sections \ref{s3} and \ref{randombitappsec}
we consider a centered Gaussian random element $X$ that takes 
values in an infinite-dimensional separable Hilbert space
$(H,\|\cdot\|_H)$.
We define and analyze algorithms
that use random bits for the approximation of
\begin{align*}
S(f) =\E(f(X))
\end{align*}
for functionals
\[
f\colon H\to\R
\]
that are Lipschitz continuous with Lipschitz constant one, i.e., 
\begin{align*}
|f(x)-f(y)|
\leq \|x-y\|_{H}
\end{align*}
for all $x,y\in H$.
For comparison we also consider algorithms that
may use uniformly distributed random numbers from $[0,1]$ 
instead of random bits.

Let $\Lip_1$ denote the corresponding class of all such functionals $f$, and
let $\mu$ denote the distribution of 
$X$ on $H$. Of course, the output $A(f)$ of a randomized
(Monte Carlo)
algorithm $A$ on input $f \in \Lip_1$ is a random quantity, and
therefore the worst case error of $A$ on the class $\Lip_1$ is defined by
\begin{align*}
e(A,\Lip_1,\mu)
=\sup_{f\in \Lip_1}
\left(\mathrm{E} \left|S(f)-A(f)\right|^2\right)^{1/2}.
\end{align*}

Consider any increasing sequence
\[
H_1 \subseteq H_2 \subseteq \dots
\]
of finite-dimensional subspaces of $H$ such that $\dim H_n = n$ for 
$n \in \N$, and put $\tilde{H} = \bigcup_{n=1}^\infty H_n$ as well
as $H_0 = \emptyset$. 
We suppose that a randomized algorithm 
may evaluate any functional $f \in \Lip_1$ at any point
$x \in \tilde{H}$ at cost $n$, if $x \in H_n \setminus H_{n-1}$. 
Furthermore, 
algorithms are assumed to perform arithmetic operations with real
numbers exactly and to evaluate elementary functions at unit cost. 
Finally, the algorithms have access to a 
random number generator at cost one per call, and here we 
distinguish two cases. If the generator provides random bits, 
we use the term of a restricted Monte Carlo algorithm. Otherwise,
if the generator provides random numbers from $[0,1]$, the
algorithm is called a Monte Carlo algorithm.

By $\cost (A,f)$ we denote the
cost for applying the randomized algorithm $A$ to the functional
$f$, which is defined as the sum of the cost associated to every 
instruction that is carried out. Observe that $\cost(A,f)$ is a
random quantity, analogously to $A(f)$, and therefore the 
worst case cost of $A$ on the class $\Lip_1$ is defined by 
\begin{align*}
\cost(A,\Lip_1)
=\sup_{f\in \Lip_1}
\mathrm{E} (\cost(A,f)).
\end{align*}
This cost model, which is called variable subspace sampling, is
appropriate for quadrature problems on infinite-dimensional spaces,
see \cite{CDMGR09} for details and for the mild measurability
assumptions involved. The latter are obviously satisfied
for the specific algorithms to be constructed below.

We are particularly interested in multilevel Monte Carlo algorithms, see
\cite{G15} for a survey. At first we consider the setting from
Section \ref{randombitappsec} with the natural choice of
subspaces
\[
H_n = \spa \{e_1,\dots,e_n\}.
\]
It follows, in particular, that $\tilde{H}$ is a dense subspace of the 
support of $\mu$. 
In the present setting of a quadrature
problem with random bits we construct a multilevel algorithm as follows.
Let $L\in\N$ be the maximal level 
of the multilevel algorithm.
On every level $\ell = 2, \dots, L$ the algorithm involves a
fine and a coarse approximation that are based on the first $m=2^\ell$ 
and $m=2^{\ell-1}$ terms, respectively,
of the Karhunen-Lo\`eve expansion of $X$. On level $\ell=1$
we only consider the fine approximation with $m=2$ terms.
The bit numbers that
are used to approximate the random coefficients $Y_i$ of $X$ are chosen
according to Theorem \ref{n-t1}.
In this way we have two dimensions of discretization:
the truncation level for the Karhunen-Lo\`eve expansion
and the bit numbers for the approximation of the random
coefficients.

Let $N_1,\dots,N_L\in\N$ be the replication numbers on the
levels $1,\dots,L$, and let $X_{\ell,j}$ with $\ell = 1,\dots,L$
and $j=1,\dots,N_\ell$ denote independent copies of $X$.
Recall the definition of $\bp(m)\in \N^m$ and $X^{(m,\bp(m))}$
according to Theorem~\ref{n-t1}.
We study the multilevel Monte Carlo algorithm
\begin{align*}
A^{L,N_1,\dots,N_L}(f)
=\frac{1}{N_1}\sum_{j=1}^{N_1}f(X_{1,j}^{(2,\bp(2))})
+\sum_{\ell=2}^{L}
\frac{1}{N_\ell}
\sum_{j=1}^{N_\ell}
\bigl(
f(X_{\ell,j}^{(2^\ell,\bp(2^\ell))})
-f(X_{\ell,j}^{(2^{\ell-1},\bp(2^{\ell-1}))})\bigr).
\end{align*}

At first we show that this algorithm only requires 
$\sum_{\ell=1}^L N_\ell \cdot |\bp(2^\ell)|$
calls to the random number generator for random bits.
Since the $X_{\ell,j}$ are independent copies of $X$ and since 
$p_i(2^{\ell-1})\leq p_i(2^\ell)$ for $i=1,\dots,2^{\ell-1}$ and 
$\ell\in\N$, it suffices to show that the joint distribution of 
$X^{(m,\bp)}$ and $X^{(\tilde{m},\tilde{\bp})}$ can be simulated using 
$|\bp|$ random bits, where $1 \leq \tilde{m} \leq m$ as well as 
$\bp \in \mathbb{N}^{m}$ and $\tilde{\bp} \in \N^{\tilde{m}}$
with $\tilde p_i\leq p_i$ for all $i=1,\dots,\tilde m$.
Recall that $X = \sum_{i=1}^\infty \lambda_i^{1/2} \cdot Y_i \cdot e_i$.
For $i=1,\dots,m$ define
\[
U_i = T^{(p_i)} \circ \Phi(Y_i).
\]
By definition we have
\begin{align}\label{smp1}
X^{(m,\bp)} 
& = \sum_{i=1}^{m} \lambda_i^{1/2} \cdot Y^{(p_i)}_i \cdot e_i
= \sum_{i=1}^{m} \lambda_i^{1/2} \cdot \Phi^{-1} \circ 
T^{(p_i)} \circ \Phi (Y_i)  \cdot e_i \\
\notag
& = \sum_{i=1}^{m} \lambda_i^{1/2} \cdot \Phi^{-1} (U_i)  \cdot e_i.
\end{align}
Since $\tilde p_i\leq p_i$ for $i=1,\dots,\tilde m$, we 
get $T^{(\tilde p_i)}=T^{(\tilde p_i)}\circ T^{(p_i)}$ and therefore
\begin{align}\label{smp2}
X^{(\tilde m,\tilde \bp)} 
&= \sum_{i=1}^{\tilde m} \lambda_i^{1/2} \cdot Y^{(\tilde p_i)}_i \cdot e_i
= \sum_{i=1}^{\tilde m} \lambda_i^{1/2} \cdot 
\Phi^{-1} \circ T^{(\tilde p_i)} \circ 
\Phi (Y_i)  \cdot e_i \\
\notag
& = \sum_{i=1}^{\tilde m} \lambda_i^{1/2} \cdot
\Phi^{-1} \circ T^{(\tilde p_i)}(U_i)  \cdot e_i.
\end{align}
Combining \eqref{smp1} and \eqref{smp2} with the fact that $U_i$ is 
uniformly distributed on $D^{(p_i)}$
and $U_1,\ldots,U_m$ are independent, we conclude that the joint 
distribution of $X^{(m,\bp)}$ and $X^{(\tilde{m},\tilde{\bp})}$
can be simulated using $|\bp|$ random bits.

\begin{theo}\label{thmmlmc}
Let $\mu$ denote the distribution of the Gaussian random element
$X$ on $H$, assume that \eqref{n-g2} is 
satisfied,
and let $\eps\in\left]0,\exp(-2)\right]$. Choose
\[
L = L(\eps) = 
\left\lceil \frac{2}{\beta-1}\cdot 
\log_2\left(z(\eps)\right) \right\rceil
\]
with
\[
z = z(\eps) = 
1+\eps^{-1}\cdot\left(\ln(\eps^{-1})\right)^{-\alpha/2}
\]
as well as
\[
N_\ell =N_\ell(\eps) = \left\lceil 
2^{-\ell \beta/2} \cdot \ell^{-\alpha/2} \cdot
K(\eps) \right\rceil
\]
for $\ell=1,\dots,L$, where
\begin{align*}
K = K(\eps) =
\eps^{-\max\left(2,\beta/(\beta-1)\right)}
\cdot
\begin{cases}
1,
&\text{if }\beta>2,\\
(\ln(\eps^{-1}))^{\max(0,1-\alpha/2)},
&\text{if }\beta=2 \wedge  \alpha\neq 2,\\
\ln(\ln(\eps^{-1})),
&\text{if }\beta=2 \wedge  \alpha=2,\\
(\ln(\eps^{-1}))^{\tfrac{\alpha}{2(1-\beta)}},
&\text{if }\beta<2.
\end{cases}
\end{align*}
Then the random bit multilevel algorithm
$A^{(\eps)}=A^{L,N_1,\dots, N_L}$ satisfies
\begin{align*}
e(A^{(\eps)},\Lip_1,\mu)
\preceq \eps
\end{align*}
and
\begin{align*}
\cost(A^{(\eps)},\Lip_1)
\asymp
\eps^{-\max(2,2/(\beta-1))}
\cdot
\begin{cases}
1,
&\text{if }\beta>2,\\
(\ln(\eps^{-1}))^{\max(0,2-\alpha)},
&\text{if }\beta=2 \wedge  \alpha\neq 2,\\
(\ln(\ln(\eps^{-1})))^2,
&\text{if }\beta=2 \wedge  \alpha=2,\\
(\ln(\eps^{-1}))^{\tfrac{\alpha}{1-\beta}},
&\text{if }\beta<2.
\end{cases}
\end{align*}
\end{theo}

\begin{proof}
At first we consider the  cost of $A^{(\eps)}$. Theorem~\ref{n-t1} implies
$|\bp(2^{\ell})| \asymp 2^\ell$, so that
\[
\cost(A^{(\eps)},\Lip_1) \asymp \sum_{\ell=1}^L 2^\ell\cdot N_\ell.
\]
In fact the number of calls to the random number generator for random 
bits is of this order,
see the discussion directly before Theorem~\ref{thmmlmc},
and the same holds true for the number of arithmetic operations as well 
as for the cost associated to the evaluation of $f$.
Observe that 
\[
L \asymp \ln(\eps^{-1})
\]
and
\[
2^L \asymp z^{\tfrac{2}{\beta-1}}.
\]
Therefore
\[
2^{-L\beta/2}\cdot L^{-\alpha/2}
\asymp
z^{\tfrac{\beta}{1-\beta}} \cdot L^{-\alpha/2}
\asymp
\eps^{\tfrac{\beta}{\beta-1}}
\cdot
(\ln(\eps^{-1}))^{\tfrac{\alpha}{2(\beta-1)}}.
\]
We conclude that
\[
2^{-\ell \beta/2} \cdot \ell^{-\alpha/2} \cdot K
\succeq 
2^{-L \beta/2} \cdot L^{-\alpha/2} \cdot K
\succeq 1,
\]
and consequently that
\[
N_\ell \asymp 2^{-\ell \beta/2} \cdot \ell^{-\alpha/2} \cdot K
\]
both hold uniformly in $\eps$ and $\ell=1,\dots,L$.
It follows that
\begin{equation}\label{n-g10}
\cost(A^{(\eps)},\Lip_1) \asymp K \cdot \sum_{\ell=1}^L 
2^{\ell (1-\beta/2)} \cdot \ell^{-\alpha/2}. 
\end{equation}
Furthermore
\[
\sum_{\ell=1}^L
2^{\ell(1-\beta/2)}
\cdot
\ell^{-\alpha/2}
\asymp
\begin{cases}
1,
&\text{if $\beta > 2$},\\
L^{\max(0,1-\alpha/2)},
&\text{if $\beta = 2 \wedge \alpha\neq 2$},\\
\ln(1+L),
&\text{if $\beta = 2 \wedge \alpha = 2$},\\
2^{L(1-\beta/2)}
\cdot
L^{-\alpha/2},
&\text{if $\beta < 2$}.
\end{cases}
\]
Since
\[
2^{L(1-\beta/2)}
\cdot L^{-\alpha/2}
\asymp
z^{\tfrac{2-\beta}{\beta-1}} \cdot L^{-\alpha/2}
\asymp
\eps^{-\tfrac{2-\beta}{\beta-1}}
\cdot
(\ln(\eps^{-1}))^{\tfrac{\alpha}{2(1-\beta)}},
\]
we obtain
\begin{equation}\label{eq100}
\sum_{\ell=1}^L
2^{\ell(1-\beta/2)}
\cdot
\ell^{-\alpha/2}
\asymp
\begin{cases}
1,
&\text{if $\beta>2$},\\
(\ln(\eps^{-1}))^{\max(0,1-\alpha/2)},
&\text{if $\beta=2 \wedge \alpha\neq 2$},\\
\ln(\ln(\eps^{-1})),
&\text{if $\beta=2 \wedge \alpha= 2$},\\
\eps^{-\tfrac{2-\beta}{\beta-1}}
\cdot
(\ln(\eps^{-1}))^{\tfrac{\alpha}{2(1-\beta)}},
&\text{if $\beta<2$}.
\end{cases}
\end{equation}
Together with \eqref{n-g10} this yields
the asymptotic estimate for $\cost(A^{(\eps)},\Lip_1)$ as claimed.

It remains to establish the asymptotic upper bound
for the error of $A^{(\eps)}$.
Observe that
\[
\left|S(f) - \mathrm{E} (A^{(\eps)}(f))\right| = 
\left|\mathrm{E} (f(X)) - \mathrm{E} (f(X^{(2^L,\bp(2^L))}))\right|
\leq \mathrm{E} \left\|X - X^{(2^L,\bp(2^L))}\right\|_H
\]
and
\begin{align*}
\Var \left(f(X^{(2,\bp(2))})\right)
&=\Var \left(f(X^{(2,\bp(2))})-f(0)\right)\\
&\leq \mathrm{E} \left|f(X^{(2,\bp(2))})-f(0) \right|^2
\leq \mathrm{E} \left\|X^{(2,\bp(2))}\right\|^2_H
<\infty
\end{align*}
as well as
\begin{multline*}
\Var \left(f(X^{(2^\ell,\bp(2^\ell))}) - 
f(X^{(2^{\ell-1},\bp(2^{\ell-1}))})\right)\\
\leq 2\cdot \mathrm{E} \left\|X - X^{(2^\ell,\bp(2^\ell))}\right\|^2_H
+2\cdot \mathrm{E} \left\|X - X^{(2^{\ell-1},\bp(2^{\ell-1}))}\right\|^2_H
\end{multline*}
for $\ell=2,\dots,L$,
due to the Lipschitz continuity of $f \in \Lip_1$.
{}From Theorem~\ref{n-t1} we hence get
\[
\sup_{f\in \Lip_1}|S(f)-\mathrm{E} (A^{(\eps)}(f))|
\preceq
2^{-L(\beta-1)/2}
\cdot
L^{-\alpha/2} \asymp z^{-1} \cdot L^{-\alpha/2} \asymp \eps
\]
as well as
\begin{align*}
\sup_{f\in \Lip_1}
\Var(A^{(\eps)}(f))
\preceq
\sum_{\ell=1}^L
\frac{1}{N_\ell}
\cdot 2^{-\ell (\beta-1)}
\cdot \ell^{-\alpha}
\asymp
K^{-1} \cdot
\sum_{\ell=1}^L
2^{\ell (1-\beta/2)} \cdot \ell^{-\alpha/2}
\asymp \eps^2,
\end{align*}
see \eqref{eq100}, and therefore
$e(A^{(\eps)},\Lip_1,\mu) \preceq \eps$ as claimed.
\end{proof}

The same analysis applies to the setting of a Brownian bridge $X$,
as studied in Section \ref{s3}. Here we employ a multilevel
algorithm with the first $2^\ell-1$ and $2^{\ell-1}-1$ terms,
respectively, of the L\'evy-Ciesielski representation of $X$
for levels $\ell \geq 2$, and with only the first term for level
$\ell=1$. Furthermore, the bit numbers are chosen according
to Theorem~\ref{p:Bp-rate}. Theorem \ref{thmmlmc} holds true also in this
case with $\beta=2$ and $\alpha=0$.

In order to analyze the optimality of the random bit multilevel algorithm
we consider the $n$-th minimal error for the random bit quadrature
problem, which is defined by
\begin{align*}
e^\res_n(\Lip_1,\mu)
=\inf\{e(A,\Lip_1,\mu)\colon 
\text{$A$ restricted Monte Carlo algorithm, 
$\cost(A,\Lip_1)\leq n$}\}
\end{align*}
for $n \in \N$.
For comparison we also consider
\begin{align*}
e_n(\Lip_1,\mu)
=\inf\{e(A,\Lip_1,\mu)\colon 
\text{$A$ Monte Carlo algorithm, 
$\cost(A,\Lip_1)\leq n$}\}
\end{align*}

\begin{cor}\label{corquad}
The $n$-th minimal error of restricted Monte Carlo
algorithms satisfies
\begin{align*}
e^\res_n(\Lip_1,\mu)
\preceq
n^{-\min(1/2,(\beta-1)/2)}
\cdot
\begin{cases}
1,
&\text{if }\beta>2,\\
(\ln(n))^{\max(0,1-\alpha/2)},
&\text{if }\beta=2\ \wedge \alpha\neq 2,\\
\ln(\ln(n)),
&\text{if }\beta=2\ \wedge \alpha= 2,\\
(\ln(n))^{-\alpha/2},
&\text{if }\beta<2.
\end{cases}
\end{align*}
The $n$-th minimal error of Monte Carlo algorithms satisfies
\begin{align*}
n^{-(\beta-1)/2}
\cdot
(\ln(n))^{-\alpha/2}
\preceq
e_n(\Lip_1,\mu)
\end{align*}
if $\beta\leq 2$, and
\begin{align*}
\limsup_{n\to\infty}
\left(
e_n(\Lip_1,\mu)
\cdot
\big(
n^{1/2}
\cdot
(\ln(n))^{(1+\beta)/2}
\cdot
(\ln(\ln(n)))^{\alpha/2}
\big)
\right)
>0
\end{align*}
if $\beta>2$. 
\end{cor}

\begin{proof}
The first claim follows directly from Theorem~\ref{thmmlmc}.
For the proof of the second claim we consider the small ball function
\begin{align*}
\phi(\eps) =
-\ln(\mu(\{ v \in H \colon \|v\|_H \leq \eps\}))
\end{align*}
of $\mu$, where $\eps\in \left]0,\infty\right[$.
{}From \cite[Proposition~4.3]{MR2357702}
or
\cite[Proposition~11.3]{MR3024389}
and
\eqref{n-g2}
we get
\begin{align*}
\phi(\eps)
\asymp
\eps^{-2/(\beta-1)} \cdot
(\ln(\eps^{-1}))^{-\alpha/(\beta-1)}
\end{align*}
for $\eps\in \left]0,\exp(-1)\right]$.
Combining this with \cite[Thm.~10]{CDMGR09} yields the second claim.
\end{proof}

\begin{rem}
In the case $\beta<2$,
Corollary~\ref{corquad} provides sharp upper and lower
bounds, up to multiplicative constants, for the $n$-minimal errors
with $e^\res_n(\Lip_1,\mu) \asymp e_n(\Lip_1,\mu)$.
For $\beta=2$ we have sharp bounds
up to logarithmic factors,
and, in particular, a superiority of Monte Carlo algorithms over 
restricted Monte Carlo may at most be present on the level of such
logarithmic factors.
For $\beta > 2$ the bounds are sharp only
up to logarithmic factors and up to the presence of a $\limsup$
in the lower bound for $e_n(\Lip_1,\mu)$.
Note that for many infinite-dimensional quadrature problems the 
asymptotic behavior of minimal errors is only known up to logarithmic 
factors.

The lower bounds from Corollary~\ref{corquad} are also true, 
if every random bit algorithm is allowed to choose the hierarchy 
of subspaces $H_n$ on its own and,
roughly speaking, without any restriction on the randomness that 
algorithms are allowed to use.
Furthermore, the general situation of a Banach space is considered
in \cite{CDMGR09}. The upper bound from
Corollary~\ref{corquad}
improves the upper bound from \cite[Thm.~10]{CDMGR09}
in terms of powers of $\ln(n)$ or $\ln(\ln(n))$, if $\beta \neq 2$
or $\alpha \geq 1$; the bounds do coincide in the remaining cases.
For simplicity of the presentation we omit the details; instead we
refer to \cite{CDMGR09}.
\end{rem}

\appendix

\section{Asymptotic Properties of $\Phi$ and $\Phi^{-1}$}

We derive some asymptotic properties of the distribution
function $\Phi$ and the inverse distribution function $\Phi^{-1}$
of the standard normal distribution.

In the sequel, we write $f(x) \approx g(x)$ as $x \to \infty$
for two functions $f,g \colon \left]a,\infty\right[ \to \R\setminus\{0\}$
if
\[
\lim_{x \to \infty} f(x) / g(x) = 1.
\]
Analogously we define $f(x) \approx g(x)$ as 
$x \searrow 0$ and $x \nearrow 1$, respectively.

We make use of 
\begin{equation}\label{eq:1-Phi}
1 - \Phi(x) \approx x^{-1} \cdot \varphi(x)
\end{equation}
as $x \to \infty$, which is well known
and follows, e.g., from L'H\^{o}pital's Rule.

\begin{lemma}\label{a:l1}
We have
\begin{align*}
\Phi^{-1}(1-2^{-p}) \approx \sqrt{\ln 4} \cdot p^{1/2}
\end{align*}
as $p \to \infty$.
\end{lemma}

\begin{proof}
With $c = \sqrt{\ln 4}$ and $x = 1-2^{-p}$, i.e., $p = -\log_2(1-x)$,
we have to show that
\begin{align*}
\Phi^{-1}(x) \approx c \cdot (-\log_2(1-x))^{1/2}
\end{align*}
as $x \nearrow 1$.
Setting $x = \Phi(y)$ this is equivalent to
\[
y \approx c \cdot \bigl(-\log_2(1-\Phi(y))\bigr)^{1/2}
\]
as $y \to \infty$.
Due to \eqref{eq:1-Phi} and L'H\^{o}pital's Rule
we get in fact
\[
\bigl(-\log_2(1-\Phi(y))\bigr)^{1/2} 
\approx y/c
\]
as $y\to\infty$.
\end{proof}

\begin{lemma}\label{a:l:h}
Let $h$ be defined by \eqref{eq:h}. Then we have
\begin{align*}
2^{-p}\, p \cdot h\bigl(\Phi^{-1}(1-2^{-p})\bigr) 
\approx
\bigl(\sqrt{2\pi} \, \ln 4\bigr)^{-1}
\end{align*}
as $p \to \infty$.
\end{lemma}

\begin{proof}
Let $a = \Phi^{-1}(1-2^{-p})$.
Use L'H\^{o}pital's Rule, or integration 
by parts, to verify
\[
h(a) \approx a^{-1} \cdot \exp(a^2/2)
\]
as $a \to \infty$. 
Moreover, observe that
\[
p \approx a^2 / \ln 4
\]
as $p \to \infty$, see Lemma \ref{a:l1}, and note that $1-\Phi(a)=2^{-p}$.
Altogether with \eqref{eq:1-Phi}, we obtain
\[
2^{-p}\, p \cdot h(a) 
\approx
(1-\Phi(a)) \cdot a^2 / \ln 4 \cdot h(a)
\approx
\bigl(\sqrt{2\pi} \, \ln 4\bigr)^{-1}
\]
as $p\to\infty$.
\end{proof}

\begin{lemma}\label{a:l:g}
Let $g$ be defined by \eqref{eq:g}. Then we have
\begin{align*}
g\bigl(\Phi^{-1}(1 - 2^{-(p+1)})\bigr)
\approx
2^{-p} \cdot p^{-1} / \ln 4
\end{align*}
as $p \to \infty$.
\end{lemma}

\begin{proof}
Set $a = \Phi^{-1}(1 - 2^{-(p+1)})$. At first we show
\begin{equation}\label{eq:local}
g(a)
\approx
2 \cdot (1 - \Phi(a)) \cdot a^{-2}
\end{equation}
as $a \to \infty$. By \eqref{eq:1-Phi} this reduces to showing
\[
g(a)
\approx
2 \cdot \varphi(a) \cdot a^{-3}
\]
as $a \to \infty$. To this end we compute the following derivatives
\begin{align*}
g'(a) &= 
-2 \cdot
\int_{[a,\infty[} (x-a) \cdot \varphi(x) \, \mathrm{d}x,\\
g''(a) &= 2 \cdot \int_{[a,\infty[} \varphi(x)\, \mathrm{d}x,\\
g'''(a) &= -2 \cdot \varphi(a)
\end{align*}
as well as
\begin{align*}
\frac{d}{da}
\left(\varphi(a) \cdot a^{-3}\right) 
&= - \varphi(a) \cdot a^{-2} - 3 \cdot \varphi(a) \cdot a^{-4},\\
\frac{d}{da}
\left(-2 \cdot \varphi(a) \cdot a^{-2}\right) 
&= 2 \cdot \varphi(a) \cdot a^{-1} + 4 \cdot \varphi(a) \cdot a^{-3},\\
\frac{d}{da}
\left(2 \cdot \varphi(a) \cdot a^{-1}\right) 
&= -2 \cdot \varphi(a) - 2 \cdot \varphi(a) \cdot a^{-2}.
\end{align*}
Using three times L'H\^{o}pital's Rule yields
\begin{align*}
\lim_{a \to \infty} \frac{g(a)}{2 \cdot \varphi(a) \cdot a^{-3}}
&= \lim_{a \to \infty} \frac{g'(a)}{-2 \cdot \varphi(a) \cdot a^{-2}}\\
&= \lim_{a \to \infty} \frac{g''(a)}{2 \cdot \varphi(a) \cdot a^{-1}}\\
&= \lim_{a \to \infty} 
\frac{g'''(a)}{
-2 \cdot \varphi(a)
} = 1.
\end{align*}
Finally, having \eqref{eq:local} at hand, Lemma~\ref{a:l1} finishes the 
proof.
\end{proof}

\begin{lemma}\label{a:l:varphi-Phi}
For $0 < x < 1$ define
\[
u(x) = \varphi\bigl(\Phi^{-1}(1-x)\bigr)
\]
and
\[
v(x) = x \cdot (\ln (x^{-1}))^{1/2}.
\]
Then we have
\[
u(x)
\approx
\sqrt{2} \cdot v(x) 
\]
as $x \searrow 0$.
\end{lemma}

\begin{proof}
Set $\tilde{u}(x) = u(x)^2$ as well as $\tilde{v}(x) = v(x)^2$. 
The derivatives read as
\begin{align*}
\tilde{u}'(x)
&=
2 \cdot \varphi\bigl(\Phi^{-1}(1-x)\bigr) \cdot \Phi^{-1}(1-x),\\
\tilde{u}''(x)
&=
2 \cdot \bigl(\Phi^{-1}(1-x)\bigr)^2 - 2
\end{align*}
as well as
\begin{align*}
\tilde{v}'(x)
&=
- 2 \cdot x \cdot \ln(x) - x,\\
\tilde{v}''(x)
&=
-2 \cdot \ln(x) - 3.
\end{align*}
Consequently, two times L'H\^{o}pital yields
\begin{align*}
\lim_{x \to 0} \frac{\tilde{v}(x)}{\tilde{u}(x)}
=
\lim_{x \to 0} \frac{-2 \cdot \ln(x) - 3}{2 \cdot (\Phi^{-1}(1-x))^2 - 2}
=
\lim_{x \to 0} \frac{\ln(x^{-1})}{(\Phi^{-1}(1-x))^2}.
\end{align*}
Now, Lemma~\ref{a:l1} yields
\begin{align*}
\bigl(\Phi^{-1}(1-x)\bigr)^2
\approx
2 \cdot \ln 2 \cdot \log_2(x^{-1})
=
2 \cdot \ln(x^{-1})
\end{align*}
as $x \searrow 0$.
Altogether we obtain
\[
\tilde{v}(x) \approx \tilde{u}(x)/2
\]
as $x \searrow 0$.
\end{proof}

\begin{lemma}\label{a:Phi-mean-val}
We have
\begin{align*}
\Phi^{-1}(b) - \Phi^{-1}(a) 
\succeq (b-a) \cdot (1-a)^{-1} \cdot (-\ln(1-a))^{-1/2}
\end{align*}
uniformly in  $1/2 < a < b < 1$.
\end{lemma}

\begin{proof}
The mean value theorem yields the existence of 
$m \in \left]a,b\right[$ such that
\begin{align*}
\Phi^{-1}(b) - \Phi^{-1}(a)
=
(b-a) \cdot (\Phi^{-1})'(m)
\geq
(b-a) \cdot \bigl(\varphi(\Phi^{-1}(a))\bigr)^{-1}.
\end{align*}
The statement is now an immediate consequence of 
Lemma~\ref{a:l:varphi-Phi}.
\end{proof}

\section*{Acknowledgment}
The authors are grateful to Steffen Omland
for many valuable discussions and contributions 
at an early stage of this project.
We thank an anonymous referee for providing us with
the reference \cite{XB17}.

Lukas Mayer is supported by the Deutsche Forschungsgemeinschaft
(DFG) within the RTG 1932 
`Stochastic Models for Innovations in the Engineering Sciences'.


\begin{thebibliography}{10}

\bibitem{Betal15a}
{\sc
C.~Brugger, C.~De~Schryver, N.~Wehn, S.~Omland, M.~Hefter, K.~Ritter,
A.~Kostiuk, R.~Korn},
{\em Mixed precision multilevel {M}onte {C}arlo on hybrid computing
systems}, 
in: 2014 IEEE Conference on Computational Intelligence for
Financial Engineering Economics (CIFEr), 2014, pp.~ 215--222.

\bibitem{C18}
{\sc J. Chevallier},
{\em Uniform decomposition of probability measures:
quantization, classification, rate of convergence}, 
arXiv:1801:02871, (2018).

\bibitem{CDMGR09}
{\sc J.~Creutzig, S.~Dereich, T.~M\"uller-Gronbach, K.~Ritter},
{\em Infinite-dimensional quadrature and approximation of
distributions},
Found.\ Comput.\ Math.\ {\bf 9} (2009), 391--429.

\bibitem{D03}
{\sc S.~Dereich},
{\em High Resolution Coding of Stochastic Processes and Small Ball
Probabilities},
Ph.D.~Thesis, FU Berlin, 2003.

\bibitem{D07a}
{\sc S.~Dereich},
{\em The coding complexity of diffusion processes under supremum norm 
distortion},
Stochastic Process.\ Appl.\ {\bf 118} (2008), 917--937.

\bibitem{D07}
{\sc S.~Dereich},
{\em The coding complexity of diffusion processes under
$L^p[0,1]$-norm distortion},
Stochastic Process.\ Appl.\ {\bf 118} (2008), 938--951.

\bibitem{D09}
{\sc S.~Dereich},
{\em Asymptotic formulae for coding problems and intermediate
optimization problems: a review},
in: Trends in Stochastic Analysis (J. Blath, P. Moerters, M.
Scheutzow, eds.),
Cambridge Univ.\ Press, Cambridge, 2009, pp.\ 187--232.

\bibitem{DFMS03}
{\sc S.~Dereich, F.~Fehringer, A.~Matoussi, M.~Scheutzow},
{\em On the link between small ball probabilities and the quantization
problem for Gaussian measures on Banach spaces}, 
J.\ Theor.\ Probab.\ {\bf 16} (2003), 249--265.

\bibitem{DS05}
{\sc S.~Dereich, M.~Scheutzow},
{\em High-resolution quantization and entropy coding for fractional
Brownian motion}, Electron.\ J.\ Probab.\ {\bf 11} (2006), 700--722.

\bibitem{DSS13}
{\sc S.~Dereich, M.~Scheutzow, R.~Schottstedt},
{\em Constructive quantization: approximation by empirical measures},
Ann.\ Inst.\ Henri Poincar\'e (B) {\bf 49} (2013), 1183--1203.

\bibitem{GYW06}
{\sc W.~Gao, P.~Ye, H.~Wang},
{\em Optimal error bound of restricted Monte Carlo in anisotropic
Sobolev classes},
Prog.\ Natur.\ Sci.\ {\bf 16} (2006), 588--593.

\bibitem{G15}
{\sc M.~B.~Giles}, 
{\em Multilevel Monte Carlo methods},
Acta Numer.\ {\bf 24} (2015),
259--328.

\bibitem{GL00}
{\sc S.~Graf, H.~Luschgy}, 
{\em Foundations of Quantization for Probability Distributions}, 
Lecture Notes in Math. {\bf 1730},
Springer, Berlin, 2000.

\bibitem{HNP04}
{\sc S. Heinrich, E.~Novak, H.~Pfeiffer},
How many random bits do we need for Monte Carlo integration?,
in: MCQMC 2002 (H. Niederreiter, ed.), 
Springer, Berlin, 2004, pp.\ 27--49.

\bibitem{MR2357702}
{\sc A. Karol', A. Nazarov, Y. Nikitin},
{\em Small ball probabilities for {G}aussian random fields and tensor
products of compact operators},
Trans.\ Amer.\ Math.\ Soc.\ {\bf 360} (2008), 1443--1474.

\bibitem{MR3024389}
{\sc M. Lifshits},
{\em Lectures on {G}aussian Processes},
Springer Briefs in Mathematics, Springer, Heidelberg, 2012.

\bibitem{LP02}
{\sc H.~Luschgy, G.~Pag\`es}, 
{\em Functional quantization of Gaussian processes},
J.\ Funct.\ Anal.\ {\bf 196} (2002), 486--531.

\bibitem{LP04}
{\sc H.~Luschgy, G.~Pag\`es}, 
{\em Sharp asymptotics for the functional quantization 
problem for Gaussian processes},
Ann.\ Appl.\ Probab.\ {\bf 32} (2004), 1574--1599.

\bibitem{LP06}
{\sc H.~Luschgy, G.~Pag\`es}, 
{\em Functional quantization of a class of Brownian diffusion: a
constructive approach},
Stochastic Process.\ Appl.\ {\bf 116} (2006), 310--336.

\bibitem{MT04}
{\sc G.~N.~Milstein, M.~V.~Tretyakov},
{\em Stochastic Numerics for Mathematical Physics},
Springer, Berlin, 2004.

\bibitem{MGR13}
{\sc T.~M\"uller-Gronbach, K.~Ritter},
{\em A local refinement strategy for constructive quantization
of scalar SDEs},
Found.\ Comput.\ Math.\ {\bf 13} (2013), 1005--1033.

\bibitem{N85}
{\sc E.~Novak},
{\em Eingeschr\"ankte Monte Carlo Verfahren zur numerischen Integration},
in: Mathematical Statistics and Applications (W. Grossman \emph{et
al.}, eds.), Reidel, Dordrecht, 1985, pp.\ 269--282.

\bibitem{N88}
{\sc E.~Novak},
{\em Deterministic and Stochastic Error Bounds in Numerical Analysis}, 
Lecture Notes in Math. {\bf 1349},
Springer, Berlin, 1988.

\bibitem{N01}
{\sc E.~Novak},
{\em Quantum complexity of integration},
J.\ Complexity {\bf 17} (2001), 2--16.

\bibitem{NP04}
{\sc E.~Novak, H.~Pfeiffer},
{\em Coin tossing algorithms for integral equations and
tractability},
Monte Carlo Methods Appl.\ {\bf 10} (2004), 491--498.

\bibitem{O16}
{\sc S. Omland},
{\em Mixed Precision Multilevel Monte Carlo Algorithms for
Reconfigurable Hardware Systems}, Ph.D.~Thesis, Verlag Dr.~Hut,
2016.

\bibitem{Betal15b}
{\sc
S.~Omland, M.~Hefter, K.~Ritter, C.~Brugger, C.~De~Schryver, N.~Wehn, A.~Kostiuk},
{\em Exploiting mixed-precision arithmetics in a multilevel {M}onte
{C}arlo approach on {FPGA}s},
in: FPGA Based Accelerators for Financial Applications
(C.~De~Schryver, ed.), Springer, Cham, 2015, pp.~191--220.

\bibitem{R00}
{\sc K.~Ritter}, 
{\em Average-Case Analysis of Numerical Problems}, 
Lecture Notes in Math. {\bf 1733},
Springer, Berlin, 2000.

\bibitem{TW92}
{\sc J.~F.~Traub, H.~Wo\'zniakowski},
{\em The Monte Carlo algorithm with a pseudorandom generator},
Math.\ Comp.\ {\bf 58}, (1992), 323--339.

\bibitem{XB17}
{\sc C.~Xu, A.~Berger},
{\em Best finite constrained approximations of one-dimensional
probabilities},
arXiv:1704:07871, (2017).

\bibitem{YH08}
{\sc P.~Ye, X.~Hu},
{\em Optimal integration error on anisotropic classes for restricted
Monte Carlo and quantum algorithms},
J.\ Approx.\ Theory {\bf 150} (2008), 24--47.

\end{thebibliography}
\end{document}